\documentclass[10pt]{amsart}

\usepackage{amsmath, amscd, amssymb, euler}
\usepackage[frame,cmtip,arrow,matrix,line,graph,curve]{xy}
\usepackage{graphpap, color}
\usepackage[mathscr]{eucal}

\numberwithin{equation}{section}

\newcommand{\CC}{\mathbb{C}}

\newcommand{\PP}{\mathbb{P}}
\newcommand{\QQ}{\mathbb{Q}}

\newcommand{\ZZ}{\mathbb{Z}}

\newcommand{\bP}{\mathbf{P}}
\newcommand{\bM}{\mathbf{M}}

\newcommand{\cal}{\mathcal}

\def\cB{{\cal B}}

\def\cE{{\cal E}}
\def\cF{{\cal F}}

\def\cO{{\cal O}}

\def\cU{{\cal U}}

\def\lra{\longrightarrow}

\def\begeq{\begin{equation}}
\def\endeq{\end{equation}}
\def\and{\quad{\rm and}\quad}

\def\and{\quad\text{and}\quad}

  \DeclareMathOperator{\Hom}{Hom}

\DeclareMathOperator{\id}{id}

\newtheorem{prop}{Proposition}[section]
\newtheorem{theo}[prop]{Theorem}
\newtheorem{lemm}[prop]{Lemma}
\newtheorem{coro}[prop]{Corollary}
\newtheorem{rema}[prop]{Remark}

\def\mzz{\overline{\mathbf{M}}_{0,0}}
\def\mzznthree{\overline{\mathbf{M}}_{0,0}(\PP^{n-1},3)}
\def\mzzntwo{\overline{\mathbf{M}}_{0,0}(\PP^{n-1},2)}
\def\mzznd{\overline{\mathbf{M}}_{0,0}(\PP^{n-1},d)}
\def\git{/\!/}

\let\lab=\label

\title{Moduli space of stable maps to projective space via GIT}
\date{}

\author{Young-Hoon Kiem}
\author{Han-Bom Moon}
\address{Department of Mathematics and Research Institute
of Mathematics, Seoul National University, Seoul 151-747, Korea}
\email{kiem@math.snu.ac.kr} \email{onespring@math.snu.ac.kr}
\thanks{Partially supported by KOSEF grant R01-2007-000-20064-0}

\begin{document}

\begin{abstract}
We compare the Kontsevich moduli space $\mzznd$ of stable maps to
projective space with the quasi-map space $\PP (
\mathrm{Sym}^d(\CC^2)\otimes \CC^n)\git SL(2)$. Consider the
birational map
\[
\bar\psi: \PP ( \mathrm{Sym}^d(\CC^2)\otimes \CC^n)\git
SL(2)\dashrightarrow \mzz (\PP^{n-1},d)
\]
which assigns to an $n$-tuple of degree $d$ homogeneous polynomials
$f_1,\cdots, f_n$ in two variables, the map
$f=(f_1:\cdots:f_n):\PP^1\to \PP^{n-1}$. In this paper, for $d=3$,
we prove that $\bar\psi$ is the composition of three blow-ups
followed by two blow-downs. Furthermore, we identify the
blow-up/down centers explicitly in terms of the moduli spaces
$\mzz(\PP^{n-1},d)$ with $d=1,2$. In particular, $\mzz(\PP^{n-1},3)$
is the $SL(2)$-quotient of a smooth rational projective variety. 
The degree two case $\mzzntwo$, which is the
blow-up of $\PP ( \mathrm{Sym}^2\CC^2\otimes \CC^n)\git SL(2)$ along
$\PP^{n-1}$, is worked out as a warm-up.
\end{abstract} \maketitle 

\section{Introduction}
The space of smooth rational curves of given degree $d$ in
projective space admits various natural moduli theoretic
compactifications via geometric invariant theory (GIT), stable maps,
Hilbert scheme or Chow scheme. Since these compactifications give us
important but different enumerative invariants, it is an interesting
problem to compare the compactifications by a sequence of explicit
blow-ups and -downs. We expect all the blow-up centers to be some
natural moduli spaces (for lower degrees). Further, we expect the
difference between the intersection numbers on any two of the moduli
theoretic compactifications should be expressed in terms of the
intersection numbers of lower degrees. If the comparison of
compactifications is completed, the variation of intersection
numbers might be calculated
by localization techniques. 
In this paper, as a
first step, we compare the GIT compactification (or quasi-map space)
and the Kontsevich moduli space of stable maps. 
The techniques we use are the Atiyah-Bott-Kirwan theory
\cite{K2,K4,K}, variation of GIT quotients \cite{DH, Tha}, the
blow-up formula \cite{GH} and construction of stable maps by
elementary modification \cite{Kiem, CCK}.

Given an $n$-tuple $(f_1,\cdots,f_n)$ of degree $d$ homogeneous
polynomials in homogeneous coordinates $t_0,t_1$ of $\PP^1$, we have
a morphism $(f_1:\cdots :f_n):\PP^1\to \PP^{n-1}$ if $f_1,\cdots,
f_n$ have no common zeros. Thus we have an $SL(2)$-invariant
rational map $ \psi_0: \PP ( \mathrm{Sym}^d(\CC^2)\otimes
\CC^n)\dashrightarrow \mzz (\PP^{n-1},d) $ which induces a
birational map
\[
\bar\psi_0: \PP ( \mathrm{Sym}^d(\CC^2)\otimes \CC^n)\git
SL(2)\dashrightarrow \mzz (\PP^{n-1},d).
\]
Our goal is to decompose $\bar\psi_0$ into a sequence of blow-ups
and blow-downs and describe the blow-up/-down centers explicitly.

When $d=1$, $\bar\psi_0$ is an isomorphism. When $d=2$, the
following is proved in \cite[\S4]{Kiem}.

\begin{theo}\label{r2th}  $\bar\psi_0$ is the inverse of a blow-up, i.e.
$\mzz(\PP^{n-1},2)$ is the blow-up of $ \PP (
\mathrm{Sym}^2(\CC^2)\otimes \CC^n)\git SL(2)$ along
$\PP(\mathrm{Sym}^2\CC^2)\times \PP^{n-1}\git SL(2)\cong
\PP^{n-1}$.\end{theo} We reproduce the proof in \S3 for reader's
convenience.

In this paper, our focus is laid on the case where $d=3$. We prove
the following in \S5.
\begin{theo}\label{thm1.2}
The birational map $\bar\psi_0$ is the composition of three blow-ups
followed by two blow-downs. The blow-up centers are respectively,
$\PP^{n-1}$, $\overline{\mathbf{M}}_{0,2}(\PP^{n-1},1)/S_2$ (where
$S_2$ interchanges the two marked points) and the blow-up of
$\overline{\mathbf{M}}_{0,1}(\PP^{n-1},2)$ along the locus of three
irreducible components. The centers of the blow-downs are
respectively the $S_2$-quotient of a $(\PP^{n-2})^2$-bundle on
$\overline{\mathbf{M}}_{0,2}(\PP^{n-1},1)$ and a
$(\PP^{n-2})^3/S_3$-bundle on $\PP^{n-1}$. In particular,
$\bar\psi_0$ is an isomorphism when $n=2$.
\end{theo}
Here of course, $S_k$ denotes the symmetric group on $k$ letters.

Let $\bP_0=\PP ( \mathrm{Sym}^3(\CC^2)\otimes \CC^n)^s$ be the
stable part of $\PP ( \mathrm{Sym}^3(\CC^2)\otimes \CC^n)$ with
respect to the action of $SL(2)$ induced from the canonical action
on $\CC^2$. Let $\bP_1$ be the blow-up of $\bP_0$ along the locus of
$n$-tuples of homogeneous polynomials having three common zeros (or,
\emph{base points}). We get $\bP_2$ by blowing up $\bP_1$ along the
proper transform of the locus of two common zeros. Let $\bP_3$ be
the blow-up of $\bP_2$ along the proper transform of the locus of
one common zero. Then we can construct a family of stable maps of
degree 3 to $\PP^{n-1}$ parameterized by $\bP_3$ by using elementary
modification. Thus we obtain an $SL(2)$-invariant morphism
\[
\psi_3:\bP_3\lra \mzznthree .
\]

The proper transform of the exceptional divisor of the second
blow-up turns out to be a $\PP^1$-bundle on a smooth variety and the
normal bundle is $\cO(-1)$ on each fiber. So, we can blow down this
divisor and obtain a complex manifold $\bP_4$. Further, the proper
transform of the exceptional divisor of the first blow-up now
becomes a $\PP^2$-bundle on a smooth variety and the normal bundle
is $\cO(-1)$ on each fiber. Hence we can blow down $\bP_4$ to obtain
a complex manifold $\bP_5$. The morphism $\psi_3$ is constant along
the fibers of the blow-downs $\bP_3\to \bP_4\to \bP_5$ and hence
factors through a holomorphic map $\psi_5:\bP_5\lra \mzznthree$
which induces
$$\bar\psi_5: \bP_5/SL(2)\lra \mzznthree .$$
We can check that this is bijective and therefore $\bar\psi_5$ is an
isomorphism of projective varieties by the Riemann existence theorem
\cite[p.442]{Hartshorne}, because $\mzznthree$ is a normal
projective variety. 
In summary, we have the following diagram.
\[
\xymatrix{ \bP_3\ar[r]^{\pi_3}\ar[d]_{\pi_4}\ar[dr]^{p_3}&
\bP_2\ar[r]^{\pi_2}&\bP_1\ar[r]^{\pi_1}&\bP_0\ar[d]\\
\bP_4\ar[d]_{\pi_5} &\bP_3/SL(2)\ar[d]\ar[rr]\ar[drr]^{\bar\psi_3} &&\bP_0/SL(2)\ar@{.>}[d]^{\bar\psi_0}\\
\bP_5\ar[r]^{p_5}& \bP_5/SL(2)\ar[rr]^{\cong}_{\bar\psi_5} &&
\mzznthree . }
\]

Theorems \ref{r2th} and \ref{thm1.2} provide us with a new way of
calculating the cohomology rings of the moduli spaces of stable
maps. By \cite{K2,K4,K}, the cohomology ring of the $SL(2)$-quotient
of a projective space is easy to determine. As an application of
Theorem \ref{r2th}, we prove the following in \S3 by the blow-up
formula of cohomology rings.
\begin{theo}
(1) The rational cohomology  ring $H^*(\mzzntwo)$ is isomorphic to
\[
\QQ [\xi,\alpha^2,\rho]/\left\langle
\frac{(\rho+2\alpha+\xi)^n-\xi^n}{\rho+2\alpha}+\frac{(\rho-2\alpha+\xi)^n-\xi^n}{\rho-2\alpha},\right.
\]
\[\left.
(\rho+2\alpha+\xi)^n+(\rho-2\alpha+\xi)^n, \xi^n\rho \right\rangle
\]
where $\xi$, $\rho$, and $\alpha^2$ are generators of degree $2,2,4$
respectively.

(2) The Poincar\'e polynomial $P_t(\mzzntwo)=\sum_{k\ge 0}t^k\dim
H^k(\mzzntwo)$ is
\[
\frac{(1-t^{2n+2})(1-t^{2n})(1-t^{2n-2})}{(1-t^2)^2(1-t^4)}.
\]

(3) The \emph{integral} Picard group of $\mzzntwo$ is
\[
\mathrm{Pic}(\mzzntwo)=\left\{\begin{matrix}\ZZ\oplus \ZZ& \text{for } & n\ge 3\\
\ZZ &\text{for } & n=2\end{matrix}\right.
\]
\end{theo}
Behrend-O'Halloran \cite[Proposition 4.27]{BOH} gave a
\emph{recursive formula} for a set of generators of the relation
ideal but \emph{closed expressions} for a set of generators were
unknown. The Poincar\'e polynomial is equivalent to the calculation
of Getzler and Pandharipande \cite{GeP}. The \emph{rational} Picard
group $\mathrm{Pic}(\mzzntwo)\otimes \QQ$ was calculated by
Pandharipande \cite{Pand}.

In \S6, we deduce the following from Theorem \ref{thm1.2}.
\begin{theo}\label{thm1.3}
(1) The Poincar\'e polynomial $P_t(\mzznthree)$ is
\[
P_t(\mzznthree)=\left(\frac{1-t^{2n+8}}{1-t^6}+2\frac{t^4-t^{2n+2}}{1-t^4}\right)
\frac{(1-t^{2n})}{(1-t^2)}\frac{(1-t^{2n})(1-t^{2n-2})}{(1-t^2)(1-t^4)}.
\]

(2) The rational cohomology ring of
$H^*(\overline{\bM}_{0,0}(\PP^\infty, 3))=\lim_{n\to
\infty}H^*(\overline{\bM}_{0,0}(\PP^{n-1}, 3))$ is isomorphic to
\[
\QQ[\xi,\alpha^2,\rho_1^3,\rho_2^2,\rho_3,\sigma]/\langle
\alpha^2\rho_1^3, \rho_1^3\sigma, \sigma^2-4\alpha^2\rho_3^2 \rangle
\]
where $\xi,\rho_3$ are degree 2 classes, $\sigma, \rho_2^2,
\alpha^2$ are degree 4 classes, and $\rho_1^3$ is a degree 6 class.
The rational cohomology ring of $H^*(\overline{\bM}_{0,0}(\PP^1,
3))$ is isomorphic to
\[
\QQ[\xi,\alpha^2]/\langle
\frac{(\xi+\alpha)^2(\xi+3\alpha)^2-(\xi-\alpha)^2(\xi-3\alpha)^2}{2\alpha},
\frac{(\xi+\alpha)^2(\xi+3\alpha)^2+(\xi-\alpha)^2(\xi-3\alpha)^2}{2}
\rangle .\]

(3) The Picard group of $\mzznthree$ is
\[
\mathrm{Pic}(\mzznthree)=\left\{\begin{matrix}\ZZ\oplus \ZZ& \text{for } & n\ge 3\\
\ZZ &\text{for } & n=2\end{matrix}\right.
\]
The generators are $\frac13 (H+\Delta)$ and $\Delta$ for $n\ge 3$
and $\frac13 (H+\Delta)$ for $n=2$, where $\Delta$ is the boundary
divisor of reducible curves and $H$ is the locus of stable maps
whose images meet a fixed codimension two subspace in $\PP^{n-1}$.
\end{theo}
It may be possible to find the cohomology ring of $\mzznthree$ for
all $n$ from Theorem \ref{thm1.2} but we content ourselves with the
$\PP^\infty$ case and the $\PP^1$ case in this paper. The above
description of $H^*(\overline{\bM}_{0,0}(\PP^\infty, 3))$ is
equivalent to the description of Behrend-O'Halloran \cite[Theorem
4.15]{BOH} and the Poincar\'e polynomial is equivalent to the
calculation given by Getzler and Pandharipande \cite{GeP}. The
\emph{rational} Picard group $\mathrm{Pic}(\mzznthree)\otimes \QQ$
was calculated by Pandharipande \cite{Pand} but the calculation of
the \emph{integral} Picard group seems new.

In a subsequent paper, we shall work out the case for $d=4$ and
higher. 
In \cite{JKi}, we compare the Kontsevich moduli space
$\mzz(\PP^3,3)$, the Hilbert scheme of twisted cubics and Simpson's
moduli space of stable sheaves.

\section{Preliminaries}

\subsection{Kontsevich moduli space}

The Kontsevich moduli space
$\overline{\mathbf{M}}_{0,k}(\PP^{n-1},d)$ or the moduli space of
stable maps to $\PP^{n-1}$ of genus 0 and degree $d$ with $k$ marked
points is a compactification of the space of smooth rational curves
of degree $d$ in $\PP^{n-1}$ with $k$ marked points. It is the
coarse moduli space of morphisms $f:C\to \PP^{n-1}$ of degree $d$
where $C$ are connected nodal curves of arithmetic genus 0 with $k$
nonsingular marked points $p_1,\cdots, p_k$ on $C$ such that the
automorphism groups of $(f,p_1,\cdots, p_k)$ are finite. Here, an
automorphism of a stable map means an automorphism $\phi : C \to C$
that satisfies $f \circ \phi = f$ and that fixes the marked points.
See \cite{FP} for the construction and basic facts on
$\overline{\mathbf{M}}_{0,k}(\PP^{n-1},d)$.

By \cite{FP,Kim-Pand}, $\overline{\mathbf{M}}_{0,k}(\PP^{n-1},d)$ is
a normal irreducible projective variety with at worst orbifold
singularities.


\subsection{Cohomology of blow-up}\lab{cohringblowup}
We recall a few basic facts on the cohomology ring of a blow-up
along a smooth submanifold from \cite{GH}. To begin with, we
consider the cohomology ring of a projective bundle $\pi:\PP N\to Y$
where $N$ is a vector bundle of rank $r$. Let $\rho=c_1(\cO_{\PP N}(1))$
and consider the exact sequence
\[
0\lra \cO(-1)\lra \pi^*N\lra Q\lra 0
\]
where $Q$ is the cokernel of the tautological monomorphism
$\cO(-1)\to \pi^*N$. By the Whitney formula,
$(1-\rho)(1+c_1(Q)+c_2(Q)+\cdots )=1+c_1(N)+\cdots+c_r(N)$. By
expanding, we obtain $c_r(Q) =\rho^r+\rho^{r-1}c_1(N)+\cdots
+c_r(N)=0$ because $Q$ is a vector bundle of rank $r-1$. By spectral
sequence, we see immediately that this is the only relation on
$H^*(Y)$ and $\rho$. In other words, $$H^*(\PP
N)=H^*(Y)[\rho]/\langle \rho^r+\rho^{r-1}c_1(N)+\cdots +\rho
c_{r-1}(N)+c_r(N) \rangle .$$

\noindent\textbf{Example.} Let $Y=B\CC^*=\PP^\infty$ be the
classifying space of $\CC^*$ and let
$\alpha=c_1(E\CC^*\times_{\CC^*}\CC)$ where $\CC^*$ acts on $\CC$
with weight $1$. Let $N$ be a vector space on which $\CC^*$ acts
with weights $w_1,\cdots, w_r$. Then the rational equivariant
cohomology ring of the projective space $\PP N$ is $$H^*_{\CC^*}(\PP
N)=H^*(E\CC^*\times_{\CC^*}\PP N)=\QQ[\rho,\alpha]/\langle
\prod_{i=1}^r(\rho+w_i\alpha)\rangle. $$

Let $X$ be a connected complex manifold and $\imath:Y\hookrightarrow
X$ be a smooth connected submanifold of codimension $r$. Let $\pi:\tilde{X}\to
X$ be the blow-up of $X$ along $Y$ and $\tilde{Y}=\PP_Y N$ be the
exceptional divisor where $N$ is the normal bundle of $Y$. From
\cite[p.605]{GH}, we have an isomorphism
\[
H^*(\tilde{X})\cong H^*(X)\oplus \bigoplus_{k=1}^{r-1}\rho^k
H^*(Y)
\]
of vector spaces where $\rho=c_1(\cO_{\tilde{X}}(-\tilde{Y}))$.
Suppose the Poincar\'e dual $[Y]\in H^{2r}(X)$ of $Y$ in $X$ is
nonzero. The ring structure of $H^*(\tilde{X})$ is not hard to
describe. First, $H^*(X)\cong \pi^*H^*(X)$ is a subring of
$H^*(\tilde{X})$ since $\pi^*$ is a monomorphism. For $\gamma\in
H^*(X)$ and $\rho^k\beta$ ($\beta\in H^*(Y)$, $1\le k<r$), their
product is $\gamma\cdot \rho^k\beta=\imath^*(\gamma)\beta\rho^k$
because $\rho$ is supported in $\tilde{Y}$. Finally, we have the
relation
\begin{equation}\lab{eqcoh2}
\rho^r+c_1(N)\rho^{r-1}+\cdots +c_{r-1}(N)\rho+\pi^*[Y]=0
\end{equation}
in the cohomology ring $H^*(\tilde{X})$. For the left hand side of
\eqref{eqcoh2} restricts to a relation on the projective bundle
$\tilde{Y}$ and thus it comes from a class in $H^*(X)$ which has
support in $Y$. By the Thom-Gysin sequence, we have an exact diagram
\[
\xymatrix{ H^0(Y)\ar[dr]_{\cup [Y]}\ar[r] & H^{2r}(X)\ar[r]\ar[d]
&H^{2r}(X-Y)\\
&H^{2r}(Y) }
\]
Since $[Y]\ne 0$, $\cup [Y]$ is injective and hence we obtain the
relation \eqref{eqcoh2}.

In particular, we have the following
\begin{prop}\lab{eqcohprop1}
If $\imath^*:H^*(X)\to H^*(Y)$ is surjective and $[Y]\ne 0$, then we
have an isomorphism of rings
\[
H^*(\tilde{X})=H^*(X)[\rho]/\langle \rho\cdot
\mathrm{ker}(\imath^*), \rho^r+c_1(N)\rho^{r-1}+\cdots
+c_{r-1}(N)\rho+\pi^*[Y] \rangle .
\]
\end{prop}

\noindent\textbf{Remark.} Everything in this subsection holds true
for equivariant cohomology rings when there is a group action on $X$
preserving $Y$. For we can simply replace $X$ and $Y$ by the
homotopy quotients $X_G=EG\times_G X$ and $Y_G=EG\times_GY$
respectively, where $EG$ is a contractible free $G$-space and
$BG=EG/G$. Recall that $H^*_G(X)=H^*(X_G)$ by definition.

\subsection{Atiyah-Bott-Kirwan theory}\lab{subsectionABK}

Let $X$ be a smooth projective variety on which a complex reductive
group $G$ acts linearly, i.e. $X\subset \PP^N$ for some $N$ and $G$
acts via a homomorphism $G\to GL(N+1)$. We denote by $X^s$ (resp.
$X^{ss}$) the open subset of stable (resp. semistable) points in
$X$. Then there is a stratification $\{S_\beta|\beta\in \cB\}$ of
$X$ indexed by a partially ordered set $\cB$ such that $X^{ss}$ is
the open stratum $S_0$ and the Gysin sequence for the pair
$(U_\beta=X-\cup_{\gamma>\beta}S_\gamma, U_\beta-S_\beta)$ splits
into an exact sequence in rational equivariant cohomology
\[
0\to H^{j-2\lambda(\beta)}_G(S_\beta)\to H^j_G(U_\beta)\to
H^j_G(U_\beta-S_\beta)\to 0
\]
where $\lambda_\beta$ is the codimension of $S_\beta$. As a
consequence, we obtain an isomorphism
\[
H^j_G(X)\cong H^j_G(X^{ss})\oplus \bigoplus_{\beta\ne
0}H^{j-2\lambda(\beta)}_G(S_\beta)
\]
of vector spaces and an injection of rings
\begin{equation}\lab{localization}
\xymatrix{H^*_G(X)=H^*_T(X)^W\ar@{^(->}[r] &
H^*_T(X)\ar@{^(->}[r]^(.4){\oplus i_F^*} & \bigoplus_{F\in
\cF}H^*_T(F) }
\end{equation}
where $T$ is the maximal torus of $G$, $W$ the Weyl group, $\cF$ the set of $T$-fixed
components $F$, $i_F:F\hookrightarrow X$ the inclusion. Hence by
finding the image of $H^{j-2\lambda(\beta)}_G(S_\beta)$ for
$\beta\ne 0$ in $\bigoplus_{F\in \cF}H^*_T(F)$, we can calculate the
kernel of the surjective homomorphism
\[
\kappa : H^*_G(X)\lra H^*_G(X^{ss})
\]
induced by the inclusion $X^{ss}\to X$. Often $H^*_G(X)$ is easy to
calculate and in that case we can calculate the cohomology ring of
$H^*_G(X^{ss})$ which is isomorphic to $H^*(X\git G)$ when
$X^{ss}=X^s$. See \cite{K2,K} for details.

\bigskip
\noindent \textbf{Example.} Let $W_d=\mathrm{Sym}^d\CC^2$ on which
$G=SL(2)$ acts in the canonical way. Then the Poincar\'e polynomial
$P_t(\PP (W_d\otimes \CC^n)\git SL(2))=\sum_{k\ge 0}t^k\dim H^k(\PP
(W_d\otimes \CC^n)\git SL(2))$ is
\[
\frac{(1-t^{2mn-2})(1-t^{2mn})}{(1-t^2)(1-t^4)}
\]
for $d=2m-1$ odd. The case $n=1$ is worked out in \cite{K2} and the
general case is straightforward. In case $d=2m$ even, the
equivariant Poincar\'e series $P_t^{SL(2)}(\PP (W_d\otimes
\CC^n)^{ss})=\sum_{k\ge 0}t^k\dim H^k_{SL(2)}(\PP (W_d\otimes
\CC^n)^{ss})$ is
\[
\frac{1-t^{2n(m+1)-2}-t^{2n(m+1)}+t^{2n(2m+1)-2}}{(1-t^2)(1-t^4)}.
\]

\bigskip
\noindent \textbf{Example.} Let $G=SL(2)$ act on $\CC^2$ in the
obvious way and trivially on $\CC^n$. Let $\bP_0$ be the semistable
part of $\bP=\PP (\mathrm{Sym}^2(\CC^2)\otimes \CC^n)$ with respect
to the action of $G$. Let us determine the equivariant cohomology
ring $H^*_G(\bP_0)$. From the previous subsection, we have an
isomorphism of rings
\[
H^*_G(\bP)=\QQ[\xi,\alpha^2]/\langle
\xi^n(\xi-2\alpha)^n(\xi+2\alpha)^n \rangle
\]
where $\xi$ is the equivariant first Chern class of $\cO(1)$ because
the weights of the action of the maximal torus $T=\CC^*$ on
$\mathrm{Sym}^2(\CC^2)\otimes \CC^n$ are $2,0,-2$, each with
multiplicity $n$. By the localization theorem, we have an inclusion
\[
i^*=(i_2^*,i_0^*,i_{-2}^*):H^*_G(\bP)\hookrightarrow
H^*_T(Z_2)\oplus H^*_T(Z_0)\oplus H^*_T(Z_{-2})
\]
where $Z_k\cong \PP^{n-1}$ for $k=2,0,-2$ are the $T$-fixed
components of weight $k$. With the identification $H^*_T(Z_k)\cong
\QQ[\xi,\alpha]/\langle \xi^n\rangle$, the homomorphism $i^*$
sends $\xi$ to $(\xi-2\alpha,\xi,\xi+2\alpha)$ and $\alpha^2$ to
$(\alpha^2,\alpha^2,\alpha^2)$. There is only one unstable stratum
$S_\beta$ in $\bP$, namely $GZ_2=GZ_{-2}$. The composition of the
Gysin map
\[
j_*:H^{*-(4n-2)}_G(S_\beta)\hookrightarrow H^*_G(\bP)
\]
with $i_2^*$ or $i_{-2}^*$ is the multiplication by the Euler class
of the normal bundle to $S_\beta$ because $Z_2,Z_{-2}\subset
S_\beta$. Hence
\[
i_{\pm 2}^*\circ j_* (1)=\frac{(\xi\mp 2\alpha)^n(\xi\mp
4\alpha)^n}{\mp 2\alpha}
\]
as the normal bundle to $Z_{\pm 2}$ in $\bP$ has Euler class
$(\xi\mp 2\alpha)^n(\xi\mp 4\alpha)^n$ and the normal bundle to
$Z_{\pm 2}$ in $S_\beta$ has Euler class $\mp 2\alpha$. Since $Z_0$
is disjoint from $S_\beta$, $i_0^*\circ j_*(1)=0$. It is easy to see
that
\begin{equation}\lab{eqcoh1}
\frac{\xi^n(\xi+2\alpha)^n}{2\alpha}+\frac{\xi^n(\xi-2\alpha)^n}{-2\alpha}
\end{equation}
is the unique element in $H^*_G(\bP)$ whose image by $i^*$ is
$(i_2^*\circ j_*(1), i_0^*\circ j_*(1), i_{-2}^*\circ j_*(1))$.
Hence, $j_*(1)$ is \eqref{eqcoh1} and similarly any element in the
image of $j_*$ is of the form
\[
f(\xi,\alpha)\frac{\xi^n(\xi+2\alpha)^n}{2\alpha}+f(\xi,-\alpha)\frac{\xi^n(\xi-2\alpha)^n}{-2\alpha}
\]
for a polynomial $f(\xi,\alpha)$. Consequently, we have an
isomorphism of rings
\begin{equation}\lab{eqcohex1}
H^*_G(\bP_0)\cong \QQ[\xi,\alpha^2]/\langle
\xi^n\frac{(\xi+2\alpha)^n-(\xi-2\alpha)^n}{2\alpha},
\xi^n\frac{(\xi+2\alpha)^n+(\xi-2\alpha)^n}{2} \rangle .
\end{equation}

\bigskip
\noindent \textbf{Example.} Let $\bP_0$ be the stable part of
$\bP=\PP (\mathrm{Sym}^3(\CC^2)\otimes \CC^n)$ with respect to the
action of $G=SL(2)$. As above, the action on $\CC^n$ is trivial and
the action of $G$ on $\CC^2$ is the obvious one. As in the previous
example, we can calculate the cohomology ring
\begin{equation}\lab{eqcohex2}
H^*_G(\bP_0)\cong \QQ[\xi,\alpha^2]/\langle
\frac{(\xi+\alpha)^n(\xi+3\alpha)^n-(\xi-\alpha)^n(\xi-3\alpha)^n}{2\alpha},\end{equation}
\[
\frac{(\xi+\alpha)^n(\xi+3\alpha)^n+(\xi-\alpha)^n(\xi-3\alpha)^n}{2}
\rangle .\] We leave the verification as an exercise.

\subsection{Stability after blow-up}

We recall a few basic results about GIT stability from
\cite[\S3]{K4}. Let $G$ be a complex reductive group acting on a
smooth variety $X$ with a linearization on an ample line bundle $L$.
Let $Y$ be a nonsingular $G$-invariant closed subvariety of $X$ and
let $\pi:\tilde{X}\to X$ be the blow-up of $X$ along $Y$ with
exceptional divisor $E$. The line bundle $L_d=\pi^*L^{\otimes
d}\otimes \cO(-E)$ is very ample for $d$ sufficiently large and the
action of $G$ on $L$ lifts to an action on $L_d$. With respect to
this linearization, we consider the (semi)stability of points in
$\tilde{X}$. We recall the following (\cite[(3.2) and (3.3)]{K4}):
\begin{enumerate}
\item If $\pi (y)$ is not semistable in $X$ then $y$ is not
semistable in $\tilde{X}$ ;
\item If $\pi (y)$ is stable in $X$ then $y$ is stable in
$\tilde{X}$.
\end{enumerate}
In particular, if $X^s=X^{ss}$, then
$\tilde{X}^s=\tilde{X}^{ss}=\pi^{-1}(X^s)$. If $X^s/G=X\git G$ is
projective, then $\mathrm{bl}_{Y^s}X^s/G=\tilde{X}\git G$ is
projective as well.

In case $X^s\ne X^{ss}$, $\pi^{-1}(X^{ss})$ is the union of some of
the strata $\tilde{S}_\beta$ described in \cite{K2}. (See
\cite[(3.4)]{K4}.) For $G=SL(2)$, the indexing $\beta$ are the
weights of the actions of the maximal torus $\CC^*$, on the fibers
of $L_d$ at $\CC^*$-fixed points in $\pi^{-1}(X^{ss})$.

\subsection{Variation of GIT quotients}\lab{ssvarGIT}

We recall the variation of Geometric Invariant Theory (GIT)
quotients from \cite{DH, Tha}. Let $G=SL(2)$ and let  $X$ be an
irreducible smooth projective variety acted on by $G$. Since $G$ is
simple, there exists at most one linearization on any ample line
bundle $L$ on $X$. Let $L_0$ and $L_1$ be two ample line bundles on
$X$ with $G$-linearizations for which semistability coincides with
stability. Let $L_t=L_0^{1-t}\otimes L_1^t$ for rational $t\in
[0,1]$. Then we have the following.

\bigskip

\noindent (1) $[0,1]$ is partitioned into subintervals
$0=t_0<t_1<\cdots <t_n=1$ such that on each subinterval
$(t_{i-1},t_i)$ for $1\le i\le n$, the GIT quotient $X\git_t G$ with
respect to $L_t$ remains constant.

\noindent (2) The walls $t_i$ are precisely those ample line bundles
$L_t$ for which there exists $x\in X$ where the maximal torus
$\CC^*$ of $G$ fixes $x$ and acts trivially on the fiber of $L_t$
over $x$.

\bigskip

Let $\tau=t_i$ be such a wall and let $X^0$ be the union of
$G$-orbits of all such $x$ as in (2) for $\tau$. Let
$\tau^{\pm}=\tau\pm \delta$ for sufficiently small $\delta>0$ and
let $L^\pm=L_{\tau^\pm}$, $L^0=L_\tau$. Further let $X^{ss}(*)$ be
the set of semistable points with respect to $L^*$ for $*=0,+,-$.
Let $X^\pm=X^{ss}(0)-X^{ss}(\mp)$ and let $X\git G(*)$ be the
quotient of $X$ with respect to $L^*$ for $*=0,+,-$. Let $v^\pm$ be
the weight of the $\CC^*$ action on the fiber of $L^\pm$ at $x\in
X^0$. Suppose $(v^+,v^-)=1$ and the stabilizer of $x$ is $\CC^*$ for
$x\in X^0$. Then we have the following.

\bigskip

\noindent (3) Let $N$ be the normal bundle to $X^0$ in $X$ and
$N^\pm$ be the positive (resp. negative) weight space of $N$. Then
$X^\pm\git G$ is the locally trivial fibration over $X^0\git G$ with
fiber weighted projective space $\PP (|w_j^\pm|)$ where $w_j^\pm$
are the weights of $N^\pm$ and $X\git G(\pm)-X^\pm\git G=X\git
G(0)-X^0\git G$.

\noindent (4) The blow-up of $X\git G(*)$ at $X^*\git G$ for
$*=0,+,-$ is the fiber product $X\git G(-)\times _{X\git G(0)}X\git
G(+)$.

\bigskip

\noindent \textbf{Example.} Let $X=\PP (\mathrm{Sym}^2\CC^2)\times
\PP (\CC^2\otimes\CC^n)$ where $G=SL(2)$ acts on $\CC^2$ and
$\mathrm{Sym}^2\CC^2$ in the obvious way and trivially on $\CC^n$.
Let us study the variation of the GIT quotient $X\git_{(1,m)}G$ as
we vary the line bundle $\cO(1,m)$. The weights of the maximal torus
$\CC^*$ on $\mathrm{Sym}^2\CC^2$ are $2,0,-2$ and the weights on
$\CC^2\otimes \CC^n$ are $1,-1$. Hence, there is only one wall,
namely at $m=2$, and $X^0$ is the union of the $G$-orbits of
$\{t_0^2\}\times \PP(t_1\otimes \CC^n)\cong \PP^{n-1}$ where
$t_0,t_1$ are homogeneous coordinates of $\PP^1$. The normal bundle
$N$ to $X^0$ has rank $n$ and the positive weight space $N^+$ has
rank $n-1$ while the negative weight space $N^-$ has rank $1$. Let
$L_0=\cO(1,1)$ and $L_1=\cO(1,3)$. Then $v^+=-1$ and $v^-=1$. So
$X^+\git G$ is a $\PP^{n-2}$-bundle on $X^0\git G=\PP^{n-1}$, namely
the projective tangent bundle $\PP T_{\PP^{n-1}}$, and $X^-\git
G=X^0\git G=\PP^{n-1}$. Therefore, $X\git _{(1,m)}G=X\git G(+)$ for
$m>>0$ is the blow-up of $X\git_{(1,1)}G=X\git G(-)$ along
$\PP^{n-1}$.

\section{Degree two case}\lab{sectiondegtwo}

In this section, we work out the degree two case as a warm-up. We
prove the following.

\begin{theo}
\begin{enumerate}
\item $\mzzntwo$ is the blow-up of $\PP (\mathrm{Sym}^2\CC^2\otimes
\CC^n)\git SL(2)$ along $\PP (\mathrm{Sym}^2\CC^2)\times \PP
(\CC^n)\git SL(2)=\PP^{n-1}$.
\item The rational cohomology  ring $H^*(\mzzntwo)$ is
\[
\QQ [\xi,\alpha^2,\rho]/\left\langle
\frac{(\rho+2\alpha+\xi)^n-\xi^n}{\rho+2\alpha}+\frac{(\rho-2\alpha+\xi)^n-\xi^n}{\rho-2\alpha},\right.
\]
\[\left.
(\rho+2\alpha+\xi)^n+(\rho-2\alpha+\xi)^n,
\xi^n\rho \right\rangle
\]
where $\xi$, $\rho$, and $\alpha^2$ are generators of degree $2,2,4$
respectively.
\item The Poincar\'e polynomial $P_t(\mzzntwo)=\sum_{k\ge 0}t^k\dim
H^k(\mzzntwo)$ is
\[
\frac{(1-t^{2n+2})(1-t^{2n})(1-t^{2n-2})}{(1-t^2)^2(1-t^4)}.
\]
\item The Picard group of $\mzzntwo$ is
\[
\mathrm{Pic}(\mzzntwo)=\left\{\begin{matrix}\ZZ\oplus \ZZ& \text{for } & n\ge 3\\
\ZZ &\text{for } & n=2\end{matrix}\right.
\]
\end{enumerate}
\end{theo}

Item (1) is Theorem 4.1 in \cite{Kiem}. We include the proof of (1)
for reader's convenience. Let $W=W_2=\mathrm{Sym}^2\CC^2$ and
$V=\CC^n$. Let $G=SL(2)$ act on $W$ in the obvious way and trivially
on $V$. An element $x$ in $\PP (W\otimes V)$ is represented by an
$n$-tuple of homogeneous quadratic polynomials in two variables
$t_0,t_1$. We call the common zero locus in $\PP^1$ of such
$n$-tuple, the \emph{base points} of $x$.

Let $\bP_0$ be the open subset of semistable points in $\PP
(W\otimes V)$ with respect to the action of $G$. Let
$\Sigma^k_0\subset \bP_0$ be the locus of $k$ base points for
$k=0,1,2$ so that we have a decomposition
\[
\bP_0=\Sigma^0_0\sqcup \Sigma^1_0\sqcup \Sigma^2_0.
\]
For $x\in\bP_0$ to have two base points, the $n$ homogeneous
polynomials representing $x$ should be all linearly dependent and
hence $\Sigma^2_0=[\PP W\times \PP V]^{ss}$ where the superscript
$ss$ denotes the semistable part with respect to $\cO(1,1)$.

Let $\pi_1:\bP_1\to \bP_0$ be the blow-up along the smooth closed
variety $\Sigma^2_0$ and let $\bP_1^s$ be the stable part of $\bP_1$
with respect to the linearization on
$\cO(1^\epsilon):=\pi_1^*\cO(1)\otimes \cO(-\epsilon E_1)$ for
sufficiently small $\epsilon >0$ where $E_1$ is the exceptional
divisor of $\pi_1$.

We claim that there is a family of stable maps to $\PP V$ of degree
$2$ parameterized by $\bP_1^s$, which gives us a $G$-invariant
morphism
\[
\psi_1:\bP_1^s\to \mzzntwo
\]
and thus a morphism $\bar\psi_1 :\bP_1^s/G\to \mzzntwo$, such that
$\bar\psi_1$ is an isomorphism on $\Sigma^0_0/G$. Note that
$\Sigma_0^0$ is the space of all holomorphic maps from $\PP^1$ to
$\PP V$ of degree 2. Since semistability coincides with stability
for $\bP_1$, $\bP_1^s/G$ is irreducible projective normal and so is
$\mzzntwo$ because $\PP^{n-1}$ is convex. Therefore, to deduce that
$\bar\psi_1$ is an isomorphism, it suffices to show $\bar\psi_1$ is
injective.

To construct a family of stable maps parameterized by $\bP_1^s$, we
start with the evaluation map $W^*\otimes (V\otimes W)\lra V$ which
gives rise to
\[
H^0(\PP^1\times \bP_0, \cO(2,1))=W\otimes (V\otimes
W)^*\longleftarrow V^*=H^0(\PP V,\cO(1)).
\]
Hence we have a rational map
\[
\varphi_0:\PP^1\times \bP_0\dashrightarrow\PP V=\PP^{n-1}
\]
which is a morphism on the open set $\PP^1\times \Sigma^0_0$. For
$x\in \Sigma_0^1$, we can choose homogeneous coordinates $t_0,t_1$
of $\PP^1$ such that $x$ is represented by an $n$-tuple of
homogeneous quadratic polynomials which are all linear combinations
of $t_0^2$ and $t_0t_1$. So $x$ is a strictly semistable point in
$\bP_0$ whose orbit closure intersects with $\Sigma^2_0$. Hence $x$
becomes unstable in $\bP_1$ (see \cite[\S6]{K4}). The proper
transform of $\Sigma^1_0$ does not appear in $\bP_1^s$.

Let $\varphi_1'$ be the composition of $\varphi_0$ and $\id\times
\pi_1:\PP^1\times \bP_1^s\to \PP^1\times \bP_0$. The two base points
of each $x\in \Sigma^2_0$ are distinct by semistability and thus the
locus in $\PP^1\times \bP_1^s$ where $\varphi_1'$ is undefined
consists of two sections over $\bP_1^s$ because $\Sigma^2_0$ is
simply connected. Let $\mu_1:\Gamma_1\to \PP^1\times \bP_1^s$ be the
blow-up along the two sections and let
$\varphi_1=\varphi_1'\circ\mu_1$.  The evaluation map above gives us
a homomorphism $V^*\otimes \cO_{\Gamma_1}\to
\mu_1^*\cO_{\PP^1\times\bP_1^s}(2,1)$ which vanishes simply along
the exceptional divisor $\cE_1$ of $\mu_1$ and hence we obtain a
surjective homomorphism
\[
V^*\otimes
\cO_{\Gamma_1}\twoheadrightarrow\mu_1^*\cO_{\PP^1\times\bP_1^s}(2,1)\otimes
\cO_{\Gamma_1}(-\cE_1).
\]
Therefore, we obtain a diagram
\[
\xymatrix{  \Gamma_1\ar[r]^f\ar[d]_{\pi} & \PP^{n-1}\\
\bP_1^s }
\]
The first map is a flat family of semistable curves and the
restriction of the second map to each fiber of $\pi$ is a degree $2$
map. Further $f$ factors through the contraction of the middle
components in $\Gamma_1$ over the exceptional divisor of $\pi_1$. So
we get a family of stable maps to $\PP^{n-1}$ of degree 2
parameterized by $\bP_1^s$ and thus a morphism $\psi_1:\bP_1^s\to
\mzzntwo$. By construction, $\psi_1$ is $G$-invariant and factors
through a morphism $\bar\psi_1:\bP_1^s/G\to \mzzntwo$. The locus
$\Sigma_0^0$ of no base points is the set of all holomorphic maps
$\PP^1\to \PP^{n-1}$ of degree two and $\bar\psi_1|_{\Sigma_0^0/G}$
is an isomorphism onto the open subset in $\mzzntwo$ of irreducible
stable maps. The complement $\mzzntwo-\bar\psi_1(\Sigma_0^0)$ is the
locus of stable maps of degree two with two irreducible components
and thus $\mzzntwo-\bar\psi_1(\Sigma_0^0)$ is a $(\PP^{n-2}\times
\PP^{n-2})/S_2$ bundle on $\PP^{n-1}$. $\PP^{n-1}$ determines the
image of the intersection point of the two irreducible components
and $(\PP^{n-2}\times \PP^{n-2})/S_2$ determines the pair of lines
in $\PP^{n-1}$ passing through the chosen point.

On the other hand, let $C=(f_j^\lambda=a_jt_0t_1+\lambda
b_jt_0^2+\lambda c_jt_1^2)_{1\le j\le n, \lambda\in \CC}$ represent
a curve in $\bP_0$ passing through $(a_jt_0t_1)_{1\le j\le n}\in
\Sigma^2_0$. Suppose $(a_j)$ is not parallel to $(b_j)$ and $(c_j)$
in $\CC^n$. Then $\varphi_0$ restricts to
$(f_1^\lambda:\cdots:f_n^\lambda):\PP^1\times C\dashrightarrow
\PP^{n-1}$ and $\Gamma_1$ over $C$ is the blow-up of $\PP^1\times C$
along $\{0,\infty\}\times \{ 0\}$. By direct local computation, we
see immediately that the morphism constructed above $\Gamma_1\to
\PP^{n-1}$ at $\lambda=0$ is the map of the tree of three $\PP^1$'s,
whose left (resp. right) component is mapped to the line joining
$(a_j)$ and $(b_j)$ (resp. $(c_j)$), and whose middle component is
mapped to the point $(a_j)$. This proves that $\bar\psi_1$ is
bijective.

\bigskip

Next we study the cohomology ring of $\mzzntwo$. By the isomorphism
$\mzzntwo\cong \bP_1^s/G$, we have an isomorphism in rational
cohomology
\[
H^*(\mzzntwo)\cong H^*_G(\bP_1^s).
\]
From \S\ref{subsectionABK}, the Poincar\'e series of
$H^*_G(\bP_0)=H^*_{SL(2)}(\PP (V\otimes W)^{ss})$ is
\[
P_t^{SL(2)}=\frac1{(1-t^2)(1-t^4)}(1-t^{4n-2}-t^{4n}+t^{6n-2})
\]
and by \eqref{eqcohex1}, we have
\begin{equation}\label{cohtwoeq}
H^*_G(\bP_0)=\QQ [\xi,\alpha^2]/\left\langle
\xi^n\frac{(\xi+2\alpha)^n-(\xi-2\alpha)^n}{2\alpha},
\xi^n\frac{(\xi+2\alpha)^n+(\xi-2\alpha)^n}{2} \right\rangle
\end{equation}
where $\xi$ and $\alpha^2$ are generators of degree 2 and 4
respectively. 

By the blow-up formula \cite[p.605]{GH}, we have
\[
P_t^{SL(2)}(\bP_1)=P_t^{SL(2)}(\bP_0)+\frac1{1-t^4}\frac{1-t^{2n}}{1-t^2}\frac{t^2-t^{4n-4}}{1-t^2}
\]
as $\Sigma^2_0/G=\PP^{n-1}$. To find the ring $H^*_G(\bP_1)$, we need
to compute the normal bundle $N$ to the blow-up center $(\PP^2\times
\PP^{n-1})^{ss}$. Let $K$ and $C$ be respectively the kernel and
cokernel of the composition
\[
\cO^{\oplus n}\twoheadrightarrow \cO(0,1)\hookrightarrow
\cO(1,1)^{\oplus 3}
\]
on $\PP^2\times \PP^{n-1}$, induced from the tautological
homomorphisms $\cO^{\oplus n}\twoheadrightarrow \cO(1)$ on
$\PP^{n-1}$ and $\cO\hookrightarrow\cO(1)^{\oplus 3}$ on $\PP^2$. By
a simple diagram chase with the Euler sequences for the projective
spaces, we see immediately that the normal bundle $N$ to
$\PP^2\times \PP^{n-1}$ in $\PP^{3n-1}$ is the bundle $\Hom (K,C)$.
By definition, the total Chern characters of $K$ and $C$ are
respectively
\[
ch(K)=n-e^{\xi_2} \and
ch(C)=e^{\xi_1+\xi_2+2\alpha}+e^{\xi_1+\xi_2}+e^{\xi_1+\xi_2-2\alpha}-e^{\xi_2}
\]
where $\xi_1=c_1(\cO(1,0))$ and $\xi_2=c_1(\cO(0,1))$. Therefore,
\[
ch(\Hom(K,C))=ch(K^*)ch(C)=(n-e^{-\xi_2})(e^{\xi_1+\xi_2+2\alpha}+e^{\xi_1+\xi_2}+e^{\xi_1+\xi_2-2\alpha}-e^{\xi_2}).
\]
If we restrict to the semistable part $(\PP^2\times \PP^{n-1})^{ss}$
where $\xi_1=0$, we obtain
\[
ch(\Hom(K,C))=e^{2\alpha}(n e^\xi-1)+e^{-2\alpha}(n e^\xi-1)
\]
with $\xi=\xi_2$ and thus the Chern classes are given by
\[
\sum_{k=0}^{2n-2}
t^kc_{2n-2-k}(N)=\frac{(t+2\alpha+\xi)^n-\xi^n}{t+2\alpha}\frac{(t-2\alpha+\xi)^n-\xi^n}{t-2\alpha}
\]
for a formal variable $t$. The restriction of the Poincar\'e dual of
the blow-up center to the exceptional divisor is the constant term
in $t$ of the above polynomial. Hence the difference between the
constant term and the Poincar\'e dual of the blow-up center is a
multiple of $\xi^n$ because $\xi^n$ generates the kernel of the
restriction homomorphism
$$H^*_G(\bP_0)\twoheadrightarrow
H^*_G((\PP^2\times \PP^{n-1})^{ss})\cong\QQ[\xi,\alpha^2]/\langle
\xi^n\rangle.$$ Therefore, by Proposition \ref{eqcohprop1}
\[
H^*_G(\bP_1)=H^*_G(\bP_0)[\rho]/\langle \xi^n\rho,
\frac{(\rho+2\alpha+\xi)^n-\xi^n}{\rho+2\alpha}\frac{(\rho-2\alpha+\xi)^n-\xi^n}{\rho-2\alpha}
+\xi^n q(\rho, \xi, \alpha)\rangle
\]
for some homogeneous polynomial $q(\rho, \xi, \alpha)$ of degree
$n-2$. Now, we subtract out the unstable part in $\bP_1$. By the
recipe of \cite{K4}, we obtain
\[
P_t^{SL(2)}(\bP_1^s)=P_t^{SL(2)}(\bP_1)-\frac1{1-t^2}\frac{1-t^{2n}}{1-t^2}\frac{t^{2n-2}(1-t^{2n-2})}{1-t^2}
\]
\begin{equation}\lab{PPtwo}
=\frac{(1-t^{2n+2})(1-t^{2n})(1-t^{2n-2})}{(1-t^2)^2(1-t^4)}.
\end{equation}
The above normal bundle $N$ splits into the direct sum of two
subbundles $N^+$ and $N^-$ with respect to the weights. Their Chern
classes are the coefficients of the polynomials in $t$
\[
\frac{(t+2\alpha+\xi)^n-\xi^n}{t+2\alpha}\and
\frac{(t-2\alpha+\xi)^n-\xi^n}{t-2\alpha}.
\]
The restrictions, of the normal bundle to the unstable stratum, to
fixed point components are given by the pullbacks of $N^+$ and $N^-$
respectively, tensored with $\cO(1)$.

Using the notation of \S2.3, the image of
$H^{j-2\lambda(\beta)}_G(S_\beta)$ for the unique unstable stratum
$S_\beta$ in $H^*_G(\bP_1)$ is generated by
\begin{equation}\lab{newrel}
\frac{(\rho+2\alpha+\xi)^n-\xi^n}{\rho+2\alpha}+\frac{(\rho-2\alpha+\xi)^n-\xi^n}{\rho-2\alpha},
(\rho+2\alpha+\xi)^n+(\rho-2\alpha+\xi)^n+c\xi^n
\end{equation}
for some $c \in \QQ$, because $\xi^n$ generates the kernel of the
restriction homomorphism
$$H^*_G(\bP_1)\longrightarrow H^*_T(\PP N^+)\oplus H^*_T(\PP N^-)$$
by direct calculation. It is an elementary exercise to check that
the two polynomials in \eqref{newrel} and $\xi^n\rho$ are pairwisely
coprime and hence the Poincar\'e polynomial of the quotient ring of
$\QQ[\xi,\alpha^2,\rho]$ by the ideal generated by the three
polynomials coincides with \eqref{PPtwo}. Therefore, the three
polynomials generate the relation ideal for $H^*_G(\bP_1^s)\cong
H^*(\mzzntwo)$. From the condition that the three polynomials
generate the relations in \eqref{cohtwoeq}, it is easy to deduce
that $c=0$. So we proved
\[
H^*_G(\bP_1^s)\cong \QQ [\xi,\alpha^2,\rho]/\left\langle
\frac{(\rho+2\alpha+\xi)^n-\xi^n}{\rho+2\alpha}+\frac{(\rho-2\alpha+\xi)^n-\xi^n}{\rho-2\alpha},\right.
\]
\[\left.
(\rho+2\alpha+\xi)^n+(\rho-2\alpha+\xi)^n,
\xi^n\rho \right\rangle
\]
as desired.

In \cite{BOH}, Behrend and O'Halloran prove that $H^*(\mzzntwo) =
\QQ[b,t,k]/(G_n)$ where $b, t$ are degree $2$ generators and $k$ is
a degree 4 generator. Here, $G_n$ denotes three polynomials defined
recursively by the matrix equation $G_n = A^{n-1}G_1$, where
\[
A=\left(\begin{array}{ccc}b&0&0\\1&0&k\\0&1&t\end{array}\right),
\qquad G_1 =
\left(\begin{array}{c}b(2b-t)\\2b-t\\2\end{array}\right).
\]
This presentation is equivalent to ours by the following change of
variables
\[
b=\xi, \quad t=2(\xi+\rho),\quad k=4\alpha^2-(\xi+\rho)^2.
\]

Finally, the Picard group of $\bP_0$ is $\ZZ$ and the group
$\mathrm{Pic}(\bP_0)^G$ of equivariant line bundles is a subgroup
because the acting group is $G=SL(2)$. Hence,
$\mathrm{Pic}(\bP_0)^G=\ZZ$. By the blow-up formula of the Picard
group \cite[II,\S8]{Hartshorne}, we obtain
$\mathrm{Pic}(\mzzntwo)\cong \mathrm{Pic}(\bP_1^s)^G=\ZZ^2$ for
$n\ge 3$ because the blow-up center is invariant. The first
isomorphism comes from Kempf's descent lemma \cite{DN}. When $n=2$,
$\mzzntwo$ is $\PP^2$ and hence the Picard group is $\ZZ$.

\section{A birational transformation}\lab{sectionbirat}

In this section, we study a birational transformation that will be
used in the subsequent section.

Let $V_i=\CC^n$ for $i=1,\cdots,r$ and let $V=\oplus_{i=1}^r V_i$.
For $z\in V$, we write $z=(z_1,\cdots, z_r)$ with $z_i\in V_i$. For
any subset $I\subset \{1,\cdots,r\}$, we let $$\bar\Sigma^I=\{z\in
V\,|\, z_i=0 \text{ for }i\in I \} \and
\bar\Sigma^k_0=\cup_{|I|=k}\bar\Sigma^I$$ for $k\le r$. We will blow
up $V$, $r$ times and then blow down $(r-1)$ times to obtain
$\prod_{i=1}^r\cO_{\PP V_i}(-1)$, a rank $r$ vector bundle on
$(\PP^{n-1})^r$.

The blow-ups are defined inductively as follows. Let $X_0=V$. For
$1\le j\le r$, let $\pi_j:X_j\to X_{j-1}$ be the blow-up of
$\bar\Sigma^{r-j+1}_{j-1}$ and let $\bar\Sigma^k_j$ be the proper
transform of $\bar\Sigma^k_{j-1}$ for $k\ne r-j+1$, while
$\bar\Sigma^{r-j+1}_j$ is the exceptional divisor of $\pi_j$. We
claim then that $X_r$ can be blown down first along
$\bar\Sigma_r^2$, next along the proper transform
$\bar\Sigma^3_{r+1}$ of $\bar\Sigma^3_r$, next along the proper
transform $\bar\Sigma^4_{r+2}$ of $\bar\Sigma^4_r$ and so on until
the blow-down along the proper transform $\bar\Sigma^r_{2r-2}$ of
$\bar\Sigma^r_r$.
\begin{prop}\lab{birprojprop}
After the $r$ blow-ups and $(r-1)$ blow-downs as above, $V$ becomes
$\prod_{i=1}^r\cO_{\PP V_i}(-1)$.
\end{prop}

\begin{proof}
We use induction on $r$. For $r=1$, there is nothing to prove.
Suppose it holds true for $r-1$. We have an open covering
$X_1=\cup_{i=1}^rU_i$ of $X_1=\cO_{\PP V}(-1)$ with
$U_i=X_1-\bar\Sigma^{\{i\}}_1$ where $\bar\Sigma^{\{i\}}_1$ is the
proper transform of $\bar\Sigma^{\{i\}}=\{z\in V\,|\,z_i=0\}$, since
$\cap_{i=1}^r\bar\Sigma^{\{i\}}_1=\emptyset$. Then
$U_i=\cO_{\PP^{n-1}}(-1)\oplus \cO_{\PP^{n-1}}(1)^{\oplus (r-1)n}$.
Over an affine open subset $\CC^{n-1}$ of $\PP^{n-1}$, $U_i$ is
$\CC^{n}\times \oplus_{i=1}^{r-1}\CC^n$ and the blow-ups described
above give us the corresponding blow-ups of
$\oplus_{i=1}^{r-1}\CC^n$. By considering an affine open cover of
$\PP^{n-1}$ and by the induction hypothesis, we see that after
$(r-1)$ more blow-ups and $(r-2)$ blow-downs, $U_i$ becomes
\[
\cO_{(\PP^{n-1})^r}(-1,\overrightarrow{0})\oplus
\bigoplus_{k=1}^{r-1} \cO_{(\PP^{n-1})^r}(1,-\overrightarrow{e_k})
\]
where $\overrightarrow{e_1},\cdots,\overrightarrow{e}_{r-1}$ are
standard basis vectors. Therefore, after $r$ blow-ups and $(r-2)$
blow-downs, $V$ becomes a smooth projective variety $X_{2r-2}$ which
admits an open covering $
\cup_{i=1}^r(\cO_{(\PP^{n-1})^r}(-1,\overrightarrow{0})\oplus
\bigoplus_{k=1}^{r-1}
\cO_{(\PP^{n-1})^r}(1,-\overrightarrow{e_k})).$ By checking the
transition maps at general points, we see that $X_{2r-2}$ is the
total space of $\cO(-1)$ over
$\PP(\oplus_{i=1}^r\cO(-\overrightarrow{e_i}))$ over
$(\PP^{n-1})^r$. Hence, $X_{2r-2}$ is the blow-up of
$\oplus_{i=1}^r\cO(-\overrightarrow{e_i})$ over $(\PP^{n-1})^r$
along the zero section. Hence we can blow down $X_{2r-2}$ further to
obtain
$$\oplus_{i=1}^r\cO_{(\PP^{n-1})^r}(-\overrightarrow{e_i})=\prod_{i=1}^r\cO_{\PP
V_i}(-1)$$ as desired.
\end{proof}

\begin{rema}{\em
More generally, let $E \to X$ be a fiber bundle locally the direct
sum of $r$ vector bundles of rank $n$. Then we can similarly define
$\bar\Sigma_0^k$ and perform $r$ blow-ups and $(r-1)$ blow-downs as
above. Proposition \ref{birprojprop} holds true in this slightly
more general setting by the same proof.}
\end{rema}

\section{Main construction}
Let $n\ge 2$ and $V=\CC^n$. Let
$W_d=\mathrm{Sym}^d\CC^2=H^0(\PP^1,\cO(d))$ be the space of
homogeneous polynomials of degree $d$ in two variables $t_0,t_1\in
H^0(\PP^1,\cO(1))$. We will frequently drop the subscript $d$ for
convenience. The identity element $\CC \to \Hom(W,W)=W\otimes W^*$
gives us a nontrivial homomorphism
\begin{equation}\lab{eq0}
H^0(\PP^1\times \PP(V\otimes W), \cO(d,1))=W\otimes (V\otimes W)^*
\leftarrow V^*=H^0(\PP V, \cO(1))
\end{equation}
and thus we obtain a rational map
\begin{equation}\lab{eq1}
\PP^1\times\PP(V\otimes W) \dashrightarrow\PP V.
\end{equation}

Let $\mathbf{P}=\PP (V\otimes W_d)\cong \PP^{(d+1)n-1}$. The
irreducible $SL(2)$-representation on $W_d$ induces a linear action
of $SL(2)$ on $\mathbf{P}$ and let $\mathbf{P}^s$ be the stable part
of $\bP$ with respect to this action in Mumford's GIT sense. Then
\eqref{eq1} restricts to a rational map
\begin{equation}\lab{eq2}
\varphi_0:\PP^1\times \bP^s\dashrightarrow\PP V
\end{equation}
which gives us a rational map $\psi_0:\bP^s\dashrightarrow \mzznd$.
Since this map is clearly $SL(2)$-invariant, we obtain a rational
map
\begin{equation}\lab{eq3}
\bar \psi_0: \PP^{(d+1)n-1}\git SL(2)=\bP^s/SL(2)
\dashrightarrow\mzznd .
\end{equation}
In this section, we prove the following.
\begin{theo}\lab{mainconstthm} For $d=3$,
the birational map $\bar\psi_0$ is the composition of three blow-ups
followed by two blow-downs. The blow-up centers are respectively,
$\PP^{n-1}$, $\overline{\mathbf{M}}_{0,2}(\PP^{n-1},1)/S_2$ (where
$S_2$ interchanges the two marked points) and the blow-up of
$\overline{\mathbf{M}}_{0,1}(\PP^{n-1},2)$ along the locus of three
irreducible components. The centers of the blow-downs are
respectively the $S_2$-quotient of a $(\PP^{n-2})^2$-bundle on
$\overline{\mathbf{M}}_{0,2}(\PP^{n-1},1)$ and a
$(\PP^{n-2})^3/S_3$-bundle on $\PP^{n-1}$.
\end{theo}

\subsection{Stratification of $\bP^s$}

An element $\xi\in \bP$ is represented by a choice of $n$ sections
of $H^0(\PP^1,\cO(d))$.  If there is no \emph{base point} (i.e.
common zero) of $\xi$, we get a regular morphism
$\varphi_0|_{\PP^1\times\{\xi\}}$ from $\PP^1$ to $\PP^{n-1}$ of
degree $d$ and thus $\psi_0:\bP^s\dashrightarrow\mzznd$ is well
defined at $\xi$.

Let us focus on the $d=3$ case, from now on. The following is
immediate from the Hilbert-Mumford criterion for stability
\cite{GIT}.
\begin{lemm}\lab{lem1} (1) Semistability coincides with stability, i.e.
$\bP^{ss}=\bP^s$.

(2) $\bP^s$ consists of $\xi\in \bP$ which has no base point of
multiplicity $\ge 2$.
\end{lemm}

Let $\bP_0=\bP^s$. We decompose $\bP^s$ by the number of base
points:
\begin{equation}\lab{eq4}
\bP_0=\Sigma^0_0\sqcup \Sigma^1_0\sqcup \cdots \sqcup \Sigma^{d}_0
\end{equation}
where $\Sigma^{k}_0$ is the locus in $\bP_0$ of $\xi$ with $k$ base
points for $k=0,1,\cdots,d$. We put the subscript $0$ is to keep
track of the blow-ups and -downs in what follows. When it is
necessary to specify the degree $d$, we shall write $\bP_j(d)$ for
$\bP_j$ and $\Sigma^{i}_j(d)$ for $\Sigma^i_j$. The rational map
$\varphi_0$ is well-defined on the open set $\PP^1\times
\Sigma_0^{0}$ and thus we have a family of degree $d$ stable maps to
$\PP^{n-1}$.

By Lemma \ref{lem1}, no element of $\bP^s$ admits a base point of
multiplicity $\ge 2$. The following proposition gives us a local
description of the stratification \eqref{eq4}.

\begin{lemm}\lab{lem2}
(1) For $k=1,2,3$, $\Sigma^{4-k}_0$ are locally closed smooth
subvarieties of $\bP_0$ whose $SL(2)$-quotients are respectively,
$\PP^{n-1}$, $\mathbf{M}_{0,2}(\PP^{n-1},1)/S_2$ and
$\mathbf{M}_{0,1}(\PP^{n-1},2)$.

(2) The normal cone of $\Sigma^3_0$ in the closure $\bar\Sigma^2_0$
is a fiber bundle locally the union of three transversal rank $n-1$ subbundles of the
normal bundle $N_{\Sigma^3_0/\bP_0}$.

(3) The normal cone of $\Sigma^3_0$ in the closure $\bar\Sigma^1_0$
is a fiber bundle locally the union of three transversal rank $2(n-1)$ subbundles of the
normal bundle $N_{\Sigma^3_0/\bP_0}$.

(4) The normal cone of $\Sigma^2_0$ in the closure $\bar\Sigma^1_0$
is a fiber bundle locally the union of two transversal rank $n-1$ subbundles of the normal
bundle $N_{\Sigma^2_0/\bP_0}$.
\end{lemm}

The deepest stratum is easy to describe. In order to have three base
points, $\xi$ must be represented by a rank 1 homomorphism in
$\PP(V\otimes W)=\PP\Hom(V^*,W)$. Since any rank 1 homomorphism
factors through $\CC$, the locus of rank 1 homomorphisms in
$\PP\Hom(V^*,W)$ is $\PP V\times \PP W$ and hence
$\Sigma^3_0=[\PP^3\times \PP^{n-1}]^s$. Since the GIT quotient $\PP
W\git SL(2)$ is just a point,
\begin{equation}
\Sigma^3_0\git SL(2)\cong \PP^{n-1}.
\end{equation}
It is important to remember that the stabilizer in
$PGL(2)=SL(2)/\{\pm 1\}$ of a point $\xi$ in $\Sigma^3_0$ is the
symmetric group $S_3$. This group will act nontrivially on the
exceptional divisors after blow-ups.

For $\Sigma^1_0$ and $\Sigma^2_0$, we consider the multiplication
morphism
\begin{equation}\lab{eq6}
\Phi^k:\PP^{k}\times \PP (\CC^n\otimes \CC^{4-k})\lra
\PP(\CC^n\otimes \CC^4)
\end{equation}
defined by $\Phi^k\left(
(a_0:\cdots:a_{k}),(b_{i,0}:\cdots:b_{i,3-k})_{1\le i\le
n}\right)\to (c_{i,j})_{1\le i\le n, 0\le j\le 3}$ where
$c_{i,j}=\sum_{l=0}^j a_lb_{i,j-l}$. Here $a_l=0$, $b_{i,j}=0$
unless $0\le l\le k$, $0\le j\le 3-k$. Then for $k\le l$, we have
$(\Phi^k)^{-1}(\Sigma^l_0(3))=[\PP^{k}\times \Sigma^{l-k}_0(3-k)]^s$
where the superscript $s$ denotes the stable part with respect to
the $SL(2)$-action on $\cO(1,1)$. Furthermore, $\Phi^k$ maps
$[\PP^{k}\times \Sigma^{0}_0 (3-k)]^s$ bijectively onto
$\Sigma^{k}_0(3)=\Sigma^k_0$. It is easy to see that the tangent map
of $\Phi^k$ is injective over the open locus of distinct (or no)
base points in $\PP (\CC^n\otimes \CC^4)$, and hence we obtain an
isomorphism $[\PP^{k}\times \Sigma^{0}_0 (3-k)]^s\cong \Sigma^{k}_0$
for $k=1,2$. Therefore
$$\Sigma^{k}_0/SL(2)=[\PP^{k}\times \Sigma^{0}_0 (3-k)]\git
SL(2)\cong \bM_{0,k}(\PP^{n-1},3-k)/S_{k}$$ because $\Sigma^{0}_0
(3-k)$ is the space of holomorphic maps $\PP^1\to \PP^{n-1}$ of
degree $3-k$.

Next, let us determine the normal cones
$C_{\Sigma^3_0/\bar\Sigma_0^1}$, $C_{\Sigma^3_0/\bar\Sigma_0^2}$ and
$C_{\Sigma^2_0/\bar\Sigma_0^1}$. 
As observed above, $\Phi^2$ and $\Phi^1$ are immersions into $\bP_0$
and they are clearly three to one over $\Sigma^3_0$. Since $\Phi^k$
are immersions, we see by local computation that the directions of
the locus $\Sigma_0^2$ of two base points in the normal bundle
$N_{\Sigma_0^3/\bP_0}$ consists of a fiber bundle locally the union
of three transversal subbundles of rank $n-1$ and that the
directions of the locus $\Sigma_0^1$ of one base point consists of a
fiber bundle locally the union of three transversal subbundles of
rank $2n-2$ whose mutual intersections are precisely the rank $n-1$
subbundles for $\Sigma_0^2$. Similarly by local computation, we
obtain that the normal cone to $\Sigma_0^2$ in $\bar\Sigma_0^1$
consists of a fiber bundle locally the union of two transversal
subbundles of rank $n-1$ of the normal bundle
$N_{\Sigma_0^2/\bP_0}$. 

\subsection{Blow-ups}

As in the previous subsection, we let $\bP_0$ be the stable part of
$\PP (V\otimes W)$ and let $\bar\Sigma_0^k$ be the closure of
$\Sigma_0^k$ in $\bP_0$. Let $\bP_1$ be the blow-up of $\bP_0$ along
the smooth subvariety $\Sigma_0^3$ and let $\bar\Sigma_1^k$ be the
proper transform of $\bar\Sigma_0^k$ for $k=1,2$. Let
$\bar\Sigma_1^3$ be the exceptional divisor of the blow-up. Then by
Lemma \ref{lem2}, $\bar\Sigma_1^2$ is a smooth subvariety of $\bP_1$
and in fact $\bar\Sigma_1^2$ is the line bundle $\cO(-1)$ over the
disjoint union of three projective bundles on $\Sigma_0^3$ in a
neighborhood of $\bar\Sigma_1^3\cap \bar\Sigma_1^2$. Furthermore,
$\bar\Sigma^1_1$ about $\bar\Sigma^1_1\cap \bar\Sigma_1^3$ is the
union of three projective subbundles of $\PP N_{\Sigma_0^3/\bP_0}$
of fiber dimension $2n-3$ whose mutual intersections are the three
projective bundles $\bar\Sigma_1^3\cap \bar\Sigma_1^2$.

Next we blow up $\bP_1$  along the smooth subvariety
$\bar\Sigma^2_1$ and obtain a variety $\bP_2$. Let $\bar\Sigma_2^k$
be the proper transform of $\bar\Sigma_1^k$ for $k=1,3$ and let
$\bar\Sigma_2^2$ be the exceptional divisor of the blow-up. Then by
Lemma \ref{lem2} again, $\bar\Sigma_2^1$ is a smooth subvariety of
$\bP_2$.

Let $\bP_3$ be the blow-up of $\bP_2$ along the smooth subvariety
$\bar\Sigma^1_2$ and let $\bar\Sigma_3^k$ be the proper transform of
$\bar\Sigma_2^k$ for $k\ne 1$. Let $\bar\Sigma_3^1$ be the
exceptional divisor of the blow-up.

By \cite[Lemma 3.11]{K4}, blow-up commutes with quotient. Therefore,
$\bP_3$ is a smooth quasi-projective variety whose quotient by the
induced $SL(2)$ action is the blow-up of $\PP(V\otimes W)\git SL(2)$
along $\Sigma_0^3/SL(2)=\PP^{n-1}$, $\bar\Sigma_1^2/SL(2)$ and
$\bar\Sigma_2^1/SL(2)$.

\begin{lemm}\lab{lem3}
(1) $\bar\Sigma_1^2/SL(2)$  is the blow-up of the GIT quotient
$\PP^2\times \PP(\CC^n\otimes \CC^2)\git_{(1,1)} SL(2)$ with respect
to the linearization $\cO(1,1)$, along $\PP^2\times \PP^1\times
\PP^{n-1}\git_{(1,1,1)}SL(2)$ with respect to the linearization
$\cO(1,1,1)$.

(2) $\bar\Sigma_2^1/SL(2)$ is the blow-up of $\PP^1\times
\PP(\CC^n\otimes \CC^3)\git _{(1,1)}SL(2)$, first along
$\PP^1\times\PP^2\times \PP^{n-1}\git_{(1,1,1)}SL(2)$ and second
along the proper transform of $\PP^1\times \PP^1\times
\PP(\CC^n\otimes \CC^2)\git_{(1,1,1)}SL(2)$.
\end{lemm}
\begin{proof}
We claim that for $k=1,2$, the morphism $$\Phi^k:[\PP^{k}\times
\PP(\CC^n\otimes \CC^{4-k})]^s\to \bar\Sigma^k_0\subset\bP_0$$
becomes an equivariant isomorphism after $3-k$ blow-ups. Since
$\bar\Sigma^k_{3-k}$ for $k=1,2$ are smooth, it suffices to show
that the induced map
\[
\Phi^k_{3-k}:[\PP^{k}\times \PP(\CC^n\otimes \CC^{4-k})]^s_{3-k}\to
\bar\Sigma^k_{3-k}\subset\bP_{3-k}
\]
is bijective where $[\PP^{k}\times \PP(\CC^n\otimes
\CC^{4-k})]^s_{3-k}$ denotes the result of $3-k$ blow-ups along
$[\PP^{k}\times \bar\Sigma^{3-k-i}_{i}(3-k)]^s$ for
$i=0,\cdots,3-k-1$. By Lemma \ref{lem2}, this is a straightforward
exercise. For instance, when $k=2$, the fiber of
$\pi_1:\bar\Sigma_1^2\to \bar\Sigma_0^2$ over a point in
$\Sigma_0^3$ is the disjoint union of three copies of $\PP^{n-2}$
which is the same as the fiber of
$\pi_1\circ\Phi^2_1=\Phi^2\circ\tilde{\pi}_1$ where $\tilde{\pi}_1$
is the blow-up map of $[\PP^2\times \PP (\CC^n\otimes \CC^2)]^s$
along the smooth subvariety $[\PP^2\times \Sigma^1_0(1)]^s$ of
codimension $n-1$, since $\Phi^2$ is three to one over $\Sigma_0^3$.

By the above claim,  $\bar\Sigma_1^2/SL(2)$ is isomorphic to the
blow-up of the GIT quotient $\PP^2\times \PP(\CC^n\otimes
\CC^2)\git_{(1,1)} SL(2)$ with respect to the linearization
$\cO(1,1)$, along $\PP^2\times \PP^1\times
\PP^{n-1}\git_{(1,1,1)}SL(2)$ with respect to the linearizaiton
$\cO(1,1,1)$. Similarly, we obtain (2).
\end{proof}

\begin{coro}\lab{stratadesc}
(1) $\bar\Sigma_1^2/SL(2)$  is
$\overline{\mathbf{M}}_{0,2}(\PP^{n-1},1)/S_2$. In other words, it
is $\PP \mathrm{Sym}^2(\cU)$ over the Grassmannian $Gr(2,n)$ where
$\cU$ is the universal rank $2$ bundle. In particular, the
Poincar\'e polynomial of $\bar\Sigma_1^2/SL(2)$  is that of
$\PP^2\times Gr(2,n)$.

(2) $\bar\Sigma_2^1/SL(2)$ is the blow-up of
$\overline{\mathbf{M}}_{0,1}(\PP^{n-1},2)$ along the locus of three
irreducible components. This locus is the $S_2$-quotient of the
fiber product $\PP (\cO^{\oplus
n}/\cO(-1))\times_{\PP^{n-1}}\PP(\cO^{\oplus n}/\cO(-1))$. In
particular, the Poincar\'e polynomial of $\bar\Sigma_2^1/SL(2)$ is
that of $\PP^1\times \PP^{n-1}\times \PP^{n-1}\times \PP^{n-2}$.
\end{coro}

\begin{proof}
(1) Note that $\PP(\CC^n\otimes \CC^2)\git SL(2)$ is the
Grassmannian $Gr(2,n)$ and the quotient $\PP^2\times
\PP(\CC^n\otimes \CC^2)\git_{(1,m)} SL(2)$ with respect to the
linearization $\cO(1,m)$ for $m>2$ is $\PP \mathrm{Sym}^2(\cU)$
over the Grassmannian $Gr(2,n)$ where $\cU$ is the universal rank 2
bundle. We study the variation of GIT quotients as we vary our
linearization from  $\cO(1,1)$ to $\cO(1,m)$ with $m>2$. (See
\cite{Tha,DH}.) There is only one wall at $m=2$ and the flip
consists of only the blow-up along $\PP^2\times \PP^1\times
\PP^{n-1}\git_{(1,1,1)}SL(2)$ with respect to the linearizaiton
$\cO(1,1,1)$. See \S\ref{ssvarGIT}.

(2) By the degree two case in \S\ref{sectiondegtwo}, we have
$$\mathrm{bl}_{\Sigma^2_0(2)}\PP(\CC^n\otimes \CC^3)\git_{1^\epsilon} SL(2)\cong \mzzntwo$$
where the linearization is $\pi_1^*\cO(1)\otimes \cO(-\epsilon E)=:
\cO(1^\epsilon)$ with $E$ the exceptional divisor of the blow-up
$\pi_1$ of $\PP(\CC^n\otimes \CC^3)$. Let
$$\bP_1(2)=\mathrm{bl}_{\Sigma^2_0(2)}\PP(\CC^n\otimes \CC^3)^s$$
and $\bP_1^s(2)$ be its stable part. Then $\PP^1\times
\bP_1^s(2)\git_{(1, m\cdot 1^\epsilon)} SL(2)$ for $m>>0$ is a
$\PP^1$-bundle over $\overline{\bM}_{0,0}(\PP^{n-1},2)$. There is
only one wall at $m_0=1/2\epsilon$ as we vary the linearization from
$(1,1^\epsilon)$ to $(1,m\cdot 1^\epsilon)$ with $m>>0$. It is
straightforward to check that the flip at $m_0$ is the composition
of a blow-up and a blow-down as follows: the blow-up is precisely
the quotient of the blow-up
$$\mu_1:\Gamma_1\lra \PP^1\times \bP_1^s(2)$$ in \S3 while the blow-down
contracts the middle components of the curves in $\Gamma_1$ lying
over $\Sigma_1^2(2)$. The result of the blow-down is obviously the
universal curve $\overline{\mathbf{M}}_{0,1}(\PP^{n-1},2)$ over
$\overline{\mathbf{M}}_{0,0}(\PP^{n-1},2)$.
\end{proof}

\subsection{Blow-downs}
Now we show that $\bP_3$ can be blown down twice. First, note that
locally the normal bundle to $\bar\Sigma_1^2$ is the direct sum of two
vector bundles, say $V_1\oplus V_2$, and the normal cone in
$\bar\Sigma_1^1$ is $V_1\cup V_2$. Hence the proper transform of
$\bar\Sigma_1^2$ in $\bP_3$ is the blow-up of $\PP(V_1\oplus V_2)$
along $\PP V_1\sqcup \PP V_2$, which is a $\PP^1$-bundle on $\PP
V_1\times_{\bar\Sigma_1^2}\PP V_2$. From \S\ref{sectionbirat}, we
can blow down along this $\PP^1$-bundle to obtain a complex manifold
$\bP_4$ with a locally free action of $SL(2)$. The proper transform
of $\bar\Sigma_1^2$ in $\bP_4$ is a $(\PP^{n-2})^2$-bundle on
$\bar\Sigma_1^2$. 

Next, the normal bundle to $\Sigma_0^3$ is locally the direct sum of
three vector bundles, say $V_1\oplus V_2\oplus V_3$ and the normal
cone in $\bar\Sigma_0^2$ is $V_1\cup V_2\cup V_3$ while the normal
cone in $\bar\Sigma_0^1$ is $\bigcup_{i\ne j}(V_i\oplus V_j)$. Hence
the proper transform of $\Sigma_0^3$ in $\bP_3$ is the blow-up of
$\PP (V_1\oplus V_2\oplus V_3)$ along $\sqcup_i\PP V_i$ and then
along $\sqcup_{i\ne j}\widetilde{\PP (V_i\oplus V_j)}$ where
$\widetilde{\PP (V_i\oplus V_j)}$ is the blow-up of $\PP (V_i\oplus
V_j)$ along $\PP V_i\sqcup \PP V_j$. From \S\ref{sectionbirat}, we
can blow down $\bP_4$ further along the proper transform
$\bar\Sigma_4^3$ of $\bar\Sigma_3^3$, to obtain a complex manifold
$\bP_5$. The image $\bar\Sigma_5^3$ in $\bP_5$ of $\bar\Sigma_4^3$
is a $(\PP^{n-2})^3$-bundle on $\Sigma_0^3$. 


\subsection{The Kontsevich moduli space as an $SL(2)$-quotient}

In this subsection, we prove that there is a family of stable maps
of degree 3 to $\PP^{n-1}$ parameterized by $\bP_3$ so that the
rational map $\psi_0:\bP_0\dashrightarrow\mzznthree$ extends to a
morphism $\psi_3: \bP_3\to \mzznthree$. Furthermore, $\psi_3$
factors through the blow-downs $\bP_3\to \bP_4\to\bP_5$ and we
obtain an $SL(2)$-invariant map $\psi_5:\bP_5\to \mzznthree$. The
induced map $\bar\psi_5:\bP_5/SL(2)\to \mzznthree$ is bijective and
hence an isomorphism of varieties because $\mzznthree$ is normal
projective.

In summary, we have the following diagram.
\[
\xymatrix{ \bP_3\ar[r]^{\pi_3}\ar[d]_{\pi_4}\ar[dr]^{p_3}&
\bP_2\ar[r]^{\pi_2}&\bP_1\ar[r]^{\pi_1}&\bP_0\ar[d]\\
\bP_4\ar[d]_{\pi_5} &\bP_3/SL(2)\ar[d]\ar[rr]\ar[drr]^{\bar\psi_3} &&\bP_0/SL(2)\ar@{.>}[d]^{\bar\psi_0}\\
\bP_5\ar[r]^{p_5}& \bP_5/SL(2)\ar[rr]^{\cong}_{\bar\psi_5} &&
\mzznthree . }
\]
\begin{prop}
$\bP_5/SL(2)$ is isomorphic to $\mzznthree$.
\end{prop}

Let $\varphi_1':\PP^1\times \bP_1\to \PP^1\times
\bP_0\dashrightarrow \PP^{n-1}$ be the composition of $\id\times
\pi_1$ and $\varphi_0$. The locus of undefined points of
$\varphi_1'$ lying over $\bar\Sigma_1^3$ consists of three sections
over $\bar\Sigma_1^3$ and let $\mu_1:\Gamma_1\to \PP^1\times \bP_1$
be the blow-up along the union of the three sections which is a
codimension 2 subvariety. Let $H_1=\mu_1^*\cO_{\PP^1\times
\bP_1}(3,1)\otimes \cO_{\Gamma_1}(-\cE_1)$ where $\cE_1$ is the
exceptional divisor of the blow-up $\mu_1$. Then the evaluation
homomorphism $ev_1':\cO_{\Gamma_1}^{\oplus n}\to
\mu_1^*\cO_{\PP^1\times \bP_1}(3,1)$ factors through
$ev_1:\cO_{\Gamma_1}^{\oplus n}\to H_1$ since $ev_1'$ vanishes along
$\cE_1$ by construction. Obviously $\Gamma_1\to \bP_1$ is a flat
family of semistable curves of genus $0$. Let
$\varphi_1=\varphi_1'\circ \mu_1$.

Next, we let $\varphi_2':\Gamma_1\times_{\bP_1}\bP_2\to \Gamma_1
\dashrightarrow \PP^{n-1}$ be the composition of $\varphi_1$ with
the obvious morphism $\Gamma_1\times_{\bP_1}\bP_2\to \Gamma_1$. Over
the divisor $\bar\Sigma_2^2$, there are two sections on
$\Gamma_1\times _{\bP_1}\bP_2$ where $\varphi_2'$ is not defined.
Let $\mu_2:\Gamma_2\to \Gamma_1\times_{\bP_1}\bP_2$ be the blow-up
along the union of these two sections. Certainly $\Gamma_2$ is a
family of semistable curves of genus 0 and the pull-back
$ev_2':\cO_{\Gamma_2}^{\oplus n}\to \mu_2^*H_1$ of $ev_1$ factors
through $ev_2: \cO_{\Gamma_2}^{\oplus n}\to H_2=\mu_2^*H_1\otimes
\cO_{\Gamma_2}(-\cE_2)$ where $\cE_2$ is the exceptional divisor of
$\mu_2$, since $ev_2'$ is vanishing along $\cE_2$. Let
$\varphi_2=\varphi_2'\circ\mu_2$.

Similarly, let $\varphi_3':\Gamma_2\times_{\bP_2}\bP_3\to \Gamma_2
\dashrightarrow \PP^{n-1}$ be the composition of $\varphi_2$ with
the obvious morphism $\Gamma_2\times_{\bP_2}\bP_3\to \Gamma_2$. Let
$\mu_3:\Gamma_3\to \Gamma_2\times_{\bP_2}\bP_3$ be the blow-up of
the section of undefined points of $\varphi_3'$ over
$\bar\Sigma_3^1$. Then $\Gamma_3$ is a family of semistable curves
of genus 0 parameterized by $\bP_3$ and the pull-back
$ev_3':\cO_{\Gamma_3}^{\oplus n}\to \mu_3^*H_2$ of $ev_2$ factors
through $ev_3: \cO_{\Gamma_3}^{\oplus n}\to H_3=\mu_3^*H_2\otimes
\cO_{\Gamma_3}(-\cE_3)$ where $\cE_3$ is the exceptional divisor of
$\mu_2$. We claim $ev_3$ is surjective and thus
$\varphi_3=\varphi_3'\circ\mu_3$ is a morphism extending
$\varphi_0$.

Indeed, we can check this by direct local computation. For example,
let $(a_j)\ne 0$ in $\CC^n$ and let $x=(a_jt_0t_1(t_0+t_1))_{1\le
j\le n}$ be a point in $\Sigma_0^3$. For $(b_j)$, $(c_j)$, $(d_j)$
in $\CC^n-\{0\}$, not parallel to $(a_j)$, consider the curve
$$C=(f_j^\lambda=a_jt_0t_1(t_0+t_1)+\lambda c_jt_0^3+\lambda
(c_j-b_j+d_j)t_0^2t_1+\lambda b_jt_1^3)_{1\le j\le n,\lambda\in\CC}$$ in
$\bP_0$ passing through $x$ at $\lambda=0$. If we restrict the above
construction to $C$, then $\Gamma_1$ at $\lambda=0$ is the comb of
four $\PP^1$'s and the restriction of $ev_1$ to each irreducible
component is respectively given by
\[
(a_jt_0+b_jt_1),\ \ \  (a_jt_0+c_jt_1),\ \ \ (a_jt_0+d_jt_1),\ \ \
(a_j)
\]
for homogeneous coordinates $t_0,t_1$ of $\PP^1$. Hence
there is no base point of $ev_1$
and the stable map to $\PP^{n-1}$ thus obtained depends only on the
three points in $\PP^{n-2}$ corresponding to $(b_j)$, $(c_j)$ and
$(d_j)$.

Next, let $(a_j), (b_j)\ne 0$ and $(a_j)$ is not parallel to $(b_j)$
in $\CC^n$. Let $x=((a_j t_0 + b_j t_1)t_0 t_1)_{1\le j\le n}$ be a
point in $\Sigma_1^2$. For $(c_j)$ which not parallel to $(a_j)$ and
for $(d_j)$ which is not parallel to $(b_j)$, consider the curve
$$D=((a_jt_0+b_jt_1)t_0t_1+\lambda c_jt_0^3+\lambda d_jt_1^3)_{1\le
j\le n, \lambda\in\CC}.$$ This represents a curve passing through
$x$ at $\lambda=0$. Then $\Gamma_2$ restricted to $\lambda=0$ is the
union of three lines and they are mapped to $\PP^{n-1}$ by
\[
(a_jt_0+b_jt_1),\ \ \ (a_jt_0+c_jt_1),\ \ \ (b_jt_0+d_jt_1).
\]
This obviously is a stable map to $\PP^{n-1}$.

Finally, suppose $(c_j)=(0)$ in the above case. This is the case
where you choose the normal direction to $\Sigma_1^2$ contained in
$\bar\Sigma_1^1$. 
Then the fiber corresponding to $\lambda=0$ in $\Gamma_2$ has three
irreducible components and the evaluation maps on the components are
respectively
\[
(a_jt_0+b_jt_1),\ \ \ (a_jt_0),\ \ \ (b_jt_0 +d_jt_1).
\]
So still there is a point (in the second component) at which the map
to $\PP^{n-1}$ is not well-defined. We choose a curve in $\bP_2$ to
this point, whose direction is normal to $\Sigma_2^1$. This amounts
to considering the limits of
$$D_\mu=((a_jt_0+b_jt_1)t_0t_1+\lambda \mu e_jt_0^3
+\lambda d_jt_1^3)_{1\le j\le n, \lambda\in\CC},\quad \mu\in \CC$$
for some $(e_j)$ not parallel to $(a_j)$. By the previous case, for
$\mu\ne 0$, the fiber of $D_\mu$ in $\Gamma_2$ at $\lambda=0$ has
three components and they are mapped to $\PP^{n-1}$ respectively by
\[
(a_jt_0+b_jt_1),\ \ \ (a_jt_0+\mu e_jt_1),\ \ \ (b_jt_0 +d_jt_1).
\]
So we have a family of nodal curves parameterized by $\mu\in \CC$
and stable maps from $E_\mu|_{\lambda=0}$ for $\mu\ne 0$. The
construction of $\Gamma_3$ blows up the only base point
$(t_0:t_1)=(0:1)$ in the second component for $\mu=0$ and $ev_3$
becomes\[ (a_jt_0+b_jt_1),\ \ \ (a_j), \ \ \ (a_jt_0+ e_jt_1),\ \ \
(b_jt_0 +d_jt_1)
\]
by the elementary modification.
%
If we contract down the constant component, then we get a stable map
of three irreducible components
\[
  (a_jt_0+b_jt_1),\ \ \ (a_jt_0+e_jt_1), \ \ \ (b_jt_0+d_jt_1).
\]

By checking case by case as above, we conclude that $ev_3$ is
surjective and $\psi_3$ factors through a holomorphic map $\psi:
\bP_5\to \mzznthree$. The bijectivity of $\bar\psi_5$ can be also
checked by case by case local computation as above. We omit the
cumbersome details.

When $n=2$, the third blow-up map $\pi_3$ is an isomorphism since
$\bar\Sigma_2^1$ is a divisor and $\pi_2$ is canceled with $\pi_4$
while $\pi_1$ is canceled with $\pi_5$. Therefore, $\bar\psi_0$ is
an isomorphism and we have $\mzz (\PP^1,3)\cong\bP_0/G=\PP
(\mathrm{Sym}^3(\CC^2)\otimes \CC^2)\git SL(2)$.

Finally, from the Riemann existence theorem \cite[p.442]{Hartshorne}
we deduce that $\bP_5/SL(2)$ is a projective variety and
$\bar\psi_5$ is an isomorphism of varieties.


\section{Cohomology ring of $\mzznthree$}

In this section, we study the cohomology of $\mzznthree$ by using
Theorem \ref{mainconstthm}.
\subsection{Betti numbers}
We compute the Poincar\'e polynomial of $\mzznthree$ in this
subsection. We use the notation $$P_t(X)=\sum t^k\dim H^k(X)\and
P_t^G(X)=\sum t^k\dim H^k_G(X)$$ for a topological space $X$. Let
$G=SL(2)$.

From \S\ref{subsectionABK}, the equivariant Poincar\'e series of
$\bP_0$ is
\[
P^G_t(\bP_0)=\frac{(1-t^{4n-2})(1-t^{4n})}{(1-t^2)(1-t^4)}.
\]
By the blow-up formula, we have
\[
P^G_t(\bP_1)=P^G_t(\bP_0)+\frac{t^2-t^{6n-6}}{1-t^2}
\frac{(1-t^{2n})}{(1-t^2)}
\]
since the blow-up center $\Sigma_0^3$ has quotient $\PP^{n-1}$.
Similarly, because $\bar\Sigma_1^2/G$ is a $\PP^2$-bundle over the
Grassmannian $Gr(2,n)$, we get
\[
P^G_t(\bP_2)=P^G_t(\bP_1)+\frac{t^2-t^{4n-4}}{1-t^2}\left(
(1+t^2+t^4)\frac{(1-t^{2n})(1-t^{2n-2})}{(1-t^2)(1-t^4)} \right) .
\]
Since $\bar\Sigma_2^1/G$ is the blow-up of
$\overline{\bM}_{0,1}(\PP^{n-1},2)$ along $\PP^{n-2}\times_{S_2}
\PP^{n-2}$ bundle on $\PP^{n-1}$, we have
\[
P^G_t(\bP_3)=P^G_t(\bP_2)+\frac{t^2-t^{2n-2}}{1-t^2}\left(
(1+t^2)\frac{(1-t^{2n})^2(1-t^{2n-2})}{(1-t^2)^3} \right) .
\]
The map $\bP_3/G\to \bP_4/G$ contracts the proper transform of
the exceptional divisor of the second blow-up, $\bar\Sigma_3^2/G$.
It is the blow-up of $\bar\Sigma_2^2/G$ along $(\bar\Sigma_2^1 \cap \bar\Sigma_2^2)/G$,
and this blow-up center is $\PP^{n-2}$ bundle over $\PP^1 \times \PP^1$ bundle over $Gr(2, n)$.
So we obtain
\[
P^G_t(\bP_4)=P^G_t(\bP_3)-\frac{t^2}{1+t^2}\left(
(1+t^2+t^4)\frac{(1-t^{2n})(1-t^{2n-2})}{(1-t^2)(1-t^4)}
\frac{(1-t^{4n-4})}{(1-t^2)}\right.
\]
\[\left.
+\frac{t^2-t^{2n-2}}{1-t^2}(1+t^2)^2
\frac{(1-t^{2n-2})(1-t^{2n})}{(1-t^2)(1-t^4)}\frac{(1-t^{2n-2})}{(1-t^2)}\right) .
\]
One should be careful about the $S_2$ action. Similarly, after the second blow-down, we obtain
\[
P^G_t(\bP_5)=P^G_t(\bP_4)-\frac{t^2+t^4}{1+t^2+t^4}
\frac{(1-t^{2n})^2(1-t^{2n-2})(1-t^{2n+2})}{(1-t^2)^3(1-t^4)} .
\]
In summary, we proved
\[
P_t(\mzznthree)=\frac{(1-t^{4n-2})(1-t^{4n})}{(1-t^2)(1-t^4)} +
\frac{t^2-t^{6n-6}}{1-t^2} \frac{(1-t^{2n})}{(1-t^2)}
\]
\[
+\frac{t^2-t^{4n-4}}{1-t^2}\left(
(1+t^2+t^4)\frac{(1-t^{2n})(1-t^{2n-2})}{(1-t^2)(1-t^4)} \right)
\]
\[
+\frac{t^2-t^{2n-2}}{1-t^2}\left(
(1+t^2)\frac{(1-t^{2n})^2(1-t^{2n-2})}{(1-t^2)^3} \right)
\]
\[
-\frac{t^2}{1+t^2}\left(
(1+t^2+t^4)\frac{(1-t^{2n})(1-t^{2n-2})}{(1-t^2)(1-t^4)}
\frac{(1-t^{4n-4})}{(1-t^2)}+\right.
\]
\[
\left.\frac{t^2-t^{2n-2}}{1-t^2}(1+t^2)^2
\frac{(1-t^{2n-2})(1-t^{2n})}{(1-t^2)(1-t^4)}\frac{(1-t^{2n-2})}{(1-t^2)}  \right)
\]
\[
-\frac{t^2+t^4}{1+t^2+t^4}
\frac{(1-t^{2n})^2(1-t^{2n-2})(1-t^{2n+2})}{(1-t^2)^3(1-t^4)}
\]
\[
=\left(\frac{1-t^{2n+8}}{1-t^6}+2\frac{t^4-t^{2n+2}}{1-t^4}\right)
\frac{(1-t^{2n})}{(1-t^2)}\frac{(1-t^{2n})(1-t^{2n-2})}{(1-t^2)(1-t^4)}.
\]

\subsection{Cohomology ring}
Since we know how to compare the cohomology rings before and after a
blow-up and we know $H^*_G(\bP_0)$ from \eqref{cohringblowup}, it
should be possible, at least in principle, to calculate the
cohomology ring of $\mzznthree$ explicitly. However it seems
extremely difficult to work out the details in reality. We content
ourselves with the limit case where $n\to \infty$, i.e.
$H^*(\overline{\bM}_{0,0}(\PP^\infty, 3))$.

First of all, by \eqref{eqcohex2}, we have
\[
H^*_G(\bP_0)\cong \QQ[\xi,\alpha^2]/\langle
\frac{(\xi+\alpha)^n(\xi+3\alpha)^n-(\xi-\alpha)^n(\xi-3\alpha)^n}{2\alpha}
\]
\[
\frac{(\xi+\alpha)^n(\xi+3\alpha)^n+(\xi-\alpha)^n(\xi-3\alpha)^n}{2}
\rangle .\]
for a degree 2 class $\xi$ and a degree 4 class $\alpha^2$. If we
take $n\to \infty$, we obtain $H^*_G(\bP_0)=\QQ[\xi,\alpha^2]$. The
first blow-up gives us a new degree 2 generator $\rho_1$ and we
obtain
\[
H^*_G(\bP_1)=\QQ[\xi,\alpha^2,\rho_1]/\langle \alpha^2\rho_1 \rangle
\]
because $\alpha^2$ generates the kernel of the surjective
restriction to the center of the first blow-up whose codimension is
infinite. Next, it is easy to check that the restriction to the
center of the second blow-up is an isomorphism and the codimension
is infinite. Hence, we have
\[
H^*_G(\bP_2)=\QQ[\xi,\alpha^2,\rho_1,\rho_2]/\langle \alpha^2\rho_1
\rangle .
\]
For the third blow-up, the restriction to the blow-up center is not
surjective any more. The blow-up center is obtained by blowing up
$[\PP^1\times \PP^{3n-1}]^s$. Let $\xi_1$ and $\xi_2$ be the
generators of $\PP^1$ and $\PP^{3n-1}$ respectively. Let $\rho_3$ be
minus the Poincar\'e dual of the exceptional divisor. Then in
addition to $\rho_3$, we need to include another generator
$\rho_3\xi_1$ of degree 4 which we denote by $\sigma$. Then by a
routine calculation, we obtain
\[
H^*_G(\bP_3)=\QQ[\xi,\alpha^2,\rho_1,\rho_2,\rho_3,\sigma]/\langle
\alpha^2\rho_1, \rho_1\sigma, \sigma^2-4\alpha^2\rho_3^2 \rangle .
\]
Now, $H^*_G(\bP_4)$ is a subring of $H^*_G(\bP_3)$. By the
description of the normal bundle to $\bar\Sigma_4^2$ in the previous
section and the blow-up formula, we see that
\[
H^*_G(\bP_4)=\QQ[\xi,\alpha^2,\rho_1,\rho_2^2,\rho_3,\sigma]/\langle
\alpha^2\rho_1, \rho_1\sigma, \sigma^2-4\alpha^2\rho_3^2 \rangle .
\]
Similarly, we obtain
\[
H^*_G(\bP_5)=\QQ[\xi,\alpha^2,\rho_1^3,\rho_2^2,\rho_3,\sigma]/\langle
\alpha^2\rho_1^3, \rho_1^3\sigma, \sigma^2-4\alpha^2\rho_3^2 \rangle
.
\]
So we proved
\[
H^*(\overline{\bM}_{0,0}(\PP^\infty,3))=\QQ[\xi,\alpha^2,\rho_1^3,\rho_2^2,\rho_3,\sigma]/\langle
\alpha^2\rho_1^3, \rho_1^3\sigma, \sigma^2-4\alpha^2\rho_3^2 \rangle
.
\]
where $\xi,\rho_3$ are degree 2 classes, $\sigma, \rho_2^2,
\alpha^2$ are degree 4 classes, and $\rho_1^3$ is a degree 6 class.
This is isomorphic to the description in \cite{BOH}.

When $n=2$, $H^*(\mzz(\PP^1,3))\cong H^*_G(\bP_0)$ because
$\bar\psi_0$ is an isomorphism. Hence $H^*(\mzz(\PP^1,3))$ is given
by \eqref{eqcohex2}.

\subsection{Picard group} We use the notation of \S5.
Since the locus of nontrivial automorphisms in $\mzznthree$ has
codimension at least two, we can delete them when calculating the
Picard group. We also delete $\bar\Sigma_5^2/G$ and
$\bar\Sigma_5^3/G$ whose codimensions are at least two. Then on the
resulting open set of $\mzznthree$, the birational map
$\bar\psi^{-1}:\mzznthree=\bP_5/G \dashrightarrow \bP_0/G$ coincides
with the blow-up map $\pi_3$ since the blow-up/down centers for
$\pi_1,\pi_2,\pi_4,\pi_5$ were deleted. Hence $\bar\psi^{-1}$ is an
honest blow-up along a smooth subvariety. For $n=2$ only,
$\bar\psi^{-1}$ is an isomorphism. By the blow-up formula for Picard
groups in \cite[II,\S8]{Hartshorne}, we obtain
\[
\mathrm{Pic}(\mzznthree)\cong \left\{
\begin{matrix}\pi_3^*\mathrm{Pic}(\bP_0/G)\oplus \ZZ\Delta & \text{ for }
 & n\ge 3\\
\pi_3^*\mathrm{Pic}(\bP_0/G) & \text{ for } & n=2\end{matrix}\right.
\]
where $\Delta=\bar\Sigma_5^1/G$ is the boundary divisor of reducible
curves. On the other hand, by Kempf's descent lemma \cite{DN} and by
checking the action of the stabilizers on the fibers of line
bundles, we obtain that the Picard group $\mathrm{Pic}(\bP_0/G)$ is
isomorphic to the equivariant Picard group
$\mathrm{Pic}(\bP_0)^G=\ZZ\mathcal{O}(2)$ which is a subgroup of
$\mathrm{Pic}(\bP_0)\cong \mathrm{Pic}(\PP^{4n-1})\cong
\ZZ\mathcal{O}(1)$.

Lastly, we note that the closure of the codimension one subset of
elements in $\bP_0$ whose images meet a fixed codimension 2 subspace
is $\mathcal{O}(6)$: Given two linear equations for the subspace, we
obtain two polynomials of degree three in $t_0,t_1\in
H^0(\PP^1,\mathcal{O}(1))$. The condition for a degree 3 curve to
meet the subspace is given by the resultant of the two polynomials
of degree 3 in $t_0,t_1$. This divisor is smooth and contains
$\Sigma^1$. Hence $\pi_3^*\mathcal{O}(6)$ is $H+\Delta$. Therefore,
$\pi_3^*\mathcal{O}(2)=\frac13(H+\Delta)$.
%
This completes the proof of Theorem \ref{thm1.3}.



\begin{thebibliography}{99}

\bibitem{BOH} K. Behrend and O'Halloran. {\em On the cohomology of
stable map spaces.} Invent. Math. {\bf 154} (2003) 385--450.

\bibitem{CCK} I. Choe, J. Choy and Y.-H. Kiem. {\em Cohomology of
the moduli space of Hecke cycles.} Topology. {\bf 44} (2005)
585--608.

\bibitem{JKi} K. Chung and Y.-H. Kiem. {\em Hilbert scheme of rational cubics via stable
maps.} In preparation.

\bibitem{DH} I. Dolgachev and Y. Hu. {\em Variation of geometric invariant theory quotients.}
Publ. IHES. {\bf 87} (1998), 5--56.

\bibitem{DN}
J.-M.~Drezet and M. Narasimhan. {\em Groupe de Picard des
vari\'et\'es de modules de fibr\'es semi-stables sur les courbes
alg\'ebriques.} Invent. Math. {\bf 97} (1989), Pages 53--94.

\bibitem{FP} W. Fulton and R. Pandharipande. {\em Notes on stable maps and
quantum cohomology}. Proceedings of Symposia in Pure Mathematics:
Algebraic Geometry Santa Cruz 1995, J. Koll¢¥ar, R. Lazarsfeld, D.
Morrison Eds., Volume 62, Part 2, 45-96.

\bibitem{GeP} E. Getzler and R. Pandharipande. {\em The Betti numbers of
$\overline{\mathcal{M}}_{0,0}(r,d)$}. J. Algebraic Geom. {\bf 15}
(2006) 709--732.

\bibitem{GH} P. Griffiths and J. Harris. {\em Principles of
Algebraic Geometry.} Wiley, New York, 1978.

\bibitem{Hartshorne} R.~Hartshorne.
\newblock {\em Algebraic geometry.}
\newblock Graduate Texts in Mathematics, No. 52. Springer-Verlag, 1977.

\bibitem{JKKW} L. Jeffrey, Y.-H. Kiem, F.C. Kirwan and J. Woolf.
{\em Cohomology pairings on singular quotients in geometric
invariant theory.} Transformation Groups, {\bf 8} (2003), 217-259.

\bibitem{Kiem} Y.-H. Kiem. {\em Hecke correspondence, stable maps
and the Kirwan desingularization}. Duke Math. J. {\bf 136} (2007) 585-618.

\bibitem{Kim-Pand} B. Kim and R. Pandharipande. {\em The connectedness of the moduli
space of maps to homogeneous spaces}. Proceedings of Symplectic
geometry and mirror symmetry, KIAS 2000, F. Fukaya, Y.-G. Oh, K.
Ono, G.Tian Eds., World Scientific (2001), 187- 203.

\bibitem{K2} F.C. Kirwan, {\em Cohomology of quotients in symplectic and
algebraic geometry.} Math. Notes vol. 31 Princeton Univ. Press,
Princeton, NJ 1985.

\bibitem{K4} F.C. Kirwan, {\em Partial desingularisations of quotients of non-singular varieties
and their Betti numbers.} Ann. Math. {\bf 122} (1985) 41-85.

\bibitem{K} F.C. Kirwan, {\em Cohomology rings of moduli spaces of
bundles over Riemann surfaces.} J. Amer. Math. Soc. {\bf 5} (1992)
853-906.

\bibitem{GIT} D. Mumford, J. Fogarty and F. Kirwan. {\em Geometric
Invariant Theory.} Ser. Modern Surveys Math. 34 Third Edition 1994.

\bibitem{MM} A. Mustata and A.M. Mustata. {\em Tautological rings of
stable map spaces.} arXiv:math/0607432.

\bibitem{Pand} R. Pandharipande. {\em Intersections of $\QQ$-divisors
on Kontsevich's moduli space $\overline M\sb {0,n}(\PP\sp r,d)$ and
enumerative geometry.} Trans. Amer. Math. Soc. {\bf 351} (1999), no.
4, 1481--1505.

\bibitem{Tha} M. Thaddeus. {\em Geometric invariant theory and
flips.} J. Amer. Math. Soc. {\bf 9} (1996), 691-723.

\end{thebibliography}
\end{document}